\documentclass[10pt,a4paper]{article} 

\usepackage[latin1]{inputenc}  
\usepackage[babel=true]{csquotes}
\usepackage{amsmath,amscd,amssymb,latexsym,graphicx,color,mathrsfs} 
 
 \usepackage[all]{xy}
\newtheorem{theorem}{Theorem}[section]

\newtheorem{proposition}[theorem]{Proposition}
\newtheorem{corollary}[theorem]{Corollary}

\newtheorem{lemma}[theorem]{Lemma}
\newtheorem{definition}[theorem]{Definition}
\newtheorem{remark}[theorem]{Remark}
\newtheorem{remarks}[theorem]{Remarks}

\newtheorem{example}[theorem]{Example}

\numberwithin{equation}{section} 
 
\newcommand{\cqfd}{\hfill{\small $\Box$}} 
 \newenvironment{proof}[1][]{{\bf Proof #1 : }}{\hfill \cqfd} 
 
\reversemarginpar

\newcommand{\Id}{\mbox{Id}} 
\newcommand{\id}{\mbox{id}} 
\newcommand{\cn}{\mathrm{cn}} 
\newcommand{\pcn}{\mathrm{pcn}}

\newcommand{\gr}{\mathscr{G}}
\newcommand{\go}{\mathscr{G} ^{(0)}}
\newcommand{\hr}{\mathscr{H}}
\newcommand{\ho}{\mathscr{H} ^{(0)}}

\def\intx{\overset{\:\circ}{X}}
\newcommand\tgt[1]{{}^{T}\kern-1pt #1}
\newcommand\adi[1]{{}^{ad}\kern-1pt #1}

\newcommand{\im}{\mathop{\mathrm{im}}\nolimits}

  \def\CC{{\mathbb{C}}}

 \def\KK{{\mathbb{K}}} 
 \def\NN{{\mathbb{N}}} 
  \def\RR{{\mathbb{R}}}

 \def\ZZ{{\mathbb{Z}}}

  \def\bR{{\mathbf{R}}}

 \def\bZ{{\mathbf{Z}}}
\def\cA{{\mathcal{A}}}  \def\cC{{\mathcal{C}}}
  \def\cF{{\mathcal{F}}}
 \def\cH{{\mathcal{H}}} 
 \def\cK{{\mathcal{K}}} 
  
\def\cP{{\mathcal{P}}}  
  \def\cU{{\mathcal{U}}}

\title{Geometric obstructions for Fredholm boundary conditions for manifolds with corners\footnote{Both authors were partially supported by ANR-14-CE25-0012-01.} \footnote{An important part of this work started during a stay of the first author at the Max Planck Institut for Mathematics at Bonn. This author wishes to express his gratitude to this institution for the excellent working conditions.}}

\author{P. Carrillo Rouse and J.M. Lescure} 
 
\date{\today} 
 
\begin{document}

\maketitle 

\begin{center}
{\bf Abstract}
\end{center}


\noindent For every connected manifold with corners we use a homology theory called conormal homology, defined in terms of faces and orientation of their conormal bundle, and whose cycles correspond geometrically to corner's cycles. Its Euler characteristic (over the rationals, dimension of the total even space minus the dimension of the total odd space), $\chi_{\cn}:=\chi_0-\chi_1$, is given by the alternated sum of the number of (open) faces of a given codimension.

\noindent The main result of the present paper is that for a compact connected manifold with corners $X$ given as a finite product of manifolds with corners of codimension less or equal to three we have that\\ 
1) If $X$ satisfies the Fredholm Perturbation property (every elliptic pseudodifferential b-operator on $X$ can be perturbed by a b-regularizing operator so it becomes Fredholm) then the even Euler corner character of $X$ vanishes, {\it i.e. $\chi_0(X)=0$}.\\
2) If the even Periodic conormal homology group vanishes, {\it i.e. $H_0^{\pcn}(X)=0$,} then $X$ satisfies the stably homotopic Fredholm Perturbation property (i.e. every elliptic pseudodifferential b-operator on $X$ satisfies the same named property up to stable homotopy among elliptic operators).\\
3) If $H_0^{\pcn}(X)$ is torsion free and if the even Euler corner character of $X$ vanishes, {\it i.e. $\chi_0(X)=0$} then $X$ satisfies the stably homotopic Fredholm Perturbation property. For example for every finite product of manifolds with corners of codimension at most two the conormal homology groups are torsion free.

\noindent The main theorem behind the above result is the explicit computation in terms of conormal homology of the $K-$theory groups of the algebra $\cK_b(X)$ of $b$-compact operators for $X$ as above. Our computation unifies the  known cases of codimension zero (smooth manifolds) and of codimension one (smooth manifolds with boundary).

\tableofcontents

\section{Introduction}

In a smooth compact manifold the vanishing of the Fredholm (Analytic) index of an elliptic (=Fredholm in this case) pseudodifferential operator is equivalent to the invertibility, up to perturabtion by a regularizing operator, of the operator. In the case of a smooth manifold with smooth boundary, not every elliptic (b-operator) totally  characteristic pseudodifferential operator is Fredholm but it can be endowed with Fredholm boundary conditions, that is it can be perturbed with a regularizing operator to become Fredholm. This non trivial fact, which goes back to Atiyah, Patodi and Singer \cite{APS1}, can also be obtained from the vanishing of a boundary analytic index (see \cite{MontNis} or below for more details). In fact, in this case the boundary analytic index takes values in the $K_0$-theory group of the algebra of regularizing operators and this K-theory group is easily  seen to vanish. Now, the case of manifolds with corners of codimension at least 2 (this includes for instance many useful domains in Euclidean spaces)  is not so well understood\footnote{We will mention some previous works at the end of this introduction}. In this paper we will show that the global topology/geometry of the corners and the way the corners form cycles enter in a fundamental way in a primary obstruction to give Fredholm boundary  conditions. As we will see the answer passes by the computation of some $K$-theory groups. We explain now with more details the problem and the content of this paper.

Using K-theoretical tools for solving index problems was the main asset in the series of papers by Atiyah-Singer (\cite{AS,AS3}) in which they introduce and prove several index formulas for smooth compact manifolds. For the case of a manifold with boundary, Atiyah-Patodi-Singer used different tools in \cite{APS1} to give a formula for the Fredholm index of a Dirac type operator with the so called APS boundary condition. It is without mentioning the importance of these results in modern mathematics. Still, besides several very interesting examples (mainly of codimension 2) and higher/more general versions of the two cases above, not too much is known in general for manifolds with corners or for other kind of manifolds with singularities. Putting an appropriate $K$-theory setting has been part of the problem for several years.

In \cite{Mel}, Melrose\footnote{See Melrose and Piazza paper \cite{MelPia} for complete details in the case with corners} constructs an algebra of pseudodifferential operators $\Psi_b^*(X)$ associated to any manifold with corners\footnote{In this paper we will always assume $X$ to be connected} $X$. The elements in this algebra are called $b-$pseudodifferential operators\footnote{To simplify we discuss only the case of scalar operators, the passage to operators acting on sections of vector bundles is done in the classic way.}, the subscript $b$ identifies these operators as obtained by "microlocalization" of the Lie algebra of $C^\infty$ vector fields on $X$ tangent to the boundary. This Lie algebra of vector fields can be explicitly obtained as sections of the so called $b$-tangent bunlde $^bTX$ (compressed tangent bundle that we will recall below).  The b-pseudodifferential calculus developed by Melrose has the classic and expected properties. In particular there is a principal symbol map
$$\sigma_b:\Psi_b^m(X)\to S^{[m]}(^bT^*X).$$  
Ellipticity has the usual meaning, namely invertibility of the principal symbol. Moreover (discussion below theorem 2.15 in \cite{MelPia}), an operator is elliptic if and only\footnote{Notice that this remark implies that to a $b$-pseudodifferential operator one can associate an "index" in the algebraic K-theory group $K_0(\Psi_b^{-\infty}(X))$ (classic construction of quasi-inverses).} if it has a quasi-inverse modulo $\Psi_b^{-\infty}(X)$. Now,  $\Psi_b^{-\infty}(X)$ contains compact operators, but also noncompact ones (as soon as $\partial X\not=\emptyset$), and compacity is there characterized by the vanishing of a suitable indicial map (p.8 ref.cit.). Elliptic $b$-pseudodifferential operators being invertible modulo compact operators -and hence Fredholm\footnote{see p.8 in \cite{MelPia} for a characterization of Fredholm operators in terms of an indicial map or \cite{Loya} thm 2.3 for the proof of Fully ellipticity iff Fredholm}-, are usually said to be  {\sl fully} elliptic. 

Now, by the property of the $b$-calculus, $\Psi_b^{0}(X)$  is included in the algebra of bounded operators on $L^2(X)$, where the $L^2$ structure is provided by some $b$-metric in the interior of $X$. We denote by $\cK_b(X)$ the norm completion of the subalgebra $\Psi_b^{-\infty}(X)$.  This $C^*$-algebra fits in a short exact sequence of $C^*$-algebras of the form
\begin{equation}\label{Introbcompact}
\xymatrix{
0\ar[r]&\cK(X)\ar[r]^-{i_0}&\cK_b(X)\ar[r]^-{r}&\cK_b(\partial X)\ar[r]&0
}
\end{equation}
where $\cK(X) $ is the algebra of compact operators in $L^2(X)$. In order to study Fredholm problems and analytic index problems one has to understand the $K$-theory of the above short exact sequence. 

To better explain how these K-theory groups enter into the study of Fredholm Perturbation properties and in order to enounce our first main results we need to settle some definitions.

{\bf Analytic and Boundary analytic Index morphism:} Given an elliptic $b$-pseudodifferential $D$, the classic construction of parametrices adapts to give a $K-$theory valued index in $K_0(\cK_b(X))$ that only depends on its principal symbol class $\sigma_b(D)\in K^0_{top}(^bT^*X)$. In more precise terms, the short exact sequence 
\begin{equation}\label{IntrobKses}
\xymatrix{
0\ar[r]&\cK_b(X)\ar[r]&\overline{\Psi_b^0(X)}\ar[r]^-{\sigma_b}&C(^bS^*X)\ar[r]&0
}
\end{equation}
gives rise to $K-$theory index morphism 
$K_1(C(^bS^*X))\to K_0(\cK_b(X))$ that factors in a canonical way by
an index morphism
\begin{equation}
\xymatrix{
K^0_{top}(^bT^*X)\ar[r]^-{Ind^a_X}&K_0(\cK_b(X))
}
\end{equation}
called {\it the Analytic Index morphism of $X$}. By composing the Analytic index with the morphism induced by the restriction to the boundary we have a morphism
\begin{equation}
\xymatrix{
K^0_{top}(^bT^*X)\ar[r]^-{Ind^\partial_X}&K_0(\cK_b(\partial X))
}
\end{equation}
called {\it the Boundary analytic index morphism of $X$}. In fact $r:K_0(\cK_b(X))\to K_0(\cK_b(\partial X))$ is an isomorphism if $\partial X\not=\emptyset$, Proposition \ref{thmbcompact}, and so the two indices above are essentially the same. In other words we completely understand the six term short exact sequence in K-theory associated to the sequence (\ref{Introbcompact}). Notice that in particular there is no contribution of the Fredholm index in the $K_0$-analytic index. 

To state the next theorem we need to define the Fredholm Perturbation Property and its stably homotopic version.

\begin{definition} Let $D\in \Psi_b^m(X)$ be elliptic. We say that $D$ satisfies: 
\begin{itemize}
\item  the {\it Fredholm Perturbation Property} $(\cF\cP)$ if there is   $R\in \Psi_b^{-\infty}(X)$ such that $D+R$ is  fully elliptic.  
\item the {\it stably homotopic Fredholm Perturbation Property} $(\cH\cF\cP)$ if there is a fully elliptic operator $D'$ with $[\sigma_b(D')]=[\sigma_b(D)]\in K_0(C^*({}^bTX))$.
\end{itemize}
\end{definition}
We also say that $X$ satisfies the   {\it (resp. stably homotopic) Fredholm Perturbation Property}  if any elliptic $b$-operator on $X$ satisfies the Fredholm property $(\cF\cP)$ (resp. $(\cH\cF\cP)$).

Property  $(\cF\cP)$ is of course stronger than property $(\cH\cF\cP)$.  In \cite{NisGauge}, Nistor characterized $(\cF\cP)$ in terms of the vanishing of an index in some particular cases. In  \cite{NSS2}, Nazaikinskii, Savin and Sternin characterized $(\cH\cF\cP)$
 for arbitrary manifolds with corners using an index map associated with their dual manifold construction. We now rephrase the result of \cite{NSS2} and we give a new proof in terms of deformation groupoids. 
\begin{theorem}\label{AnavsFredthm1intro}
Let $D$ be an elliptic $b$-pseudodifferential operator on a compact manifold with corners $X$. Then  $D$ satisfies $(\cH\cF\cP)$ if and only if 
\(
 Ind_\partial([\sigma_b(D)])=0
\).\\
In particular if $D$ satisfies $(\cF\cP)$ then its boundary analytic index vanishes. 
\end{theorem}

The above results fit exactly with the K-theory vs Index theory Atiyah-Singer's program and in that sense it is not completely unexpected. Now, in order to give a full characterization of the Fredholm perturbation property one is first led to compute or understand the $K$-theory groups for the algebras (\ref{Introbcompact}) preferably in terms of the geometry/topology of the manifold with corners. As it happens, the only previously known cases are: 
\begin{itemize}
\item the $K$-theory of the compact operators $\cK(X)$, giving $K_0(\cK(X))=\bZ$ and $K_1(\cK(X))=0$, which is of course essential for classic index theory purposes; 
\item the $K$-theory of $\cK_b(X)$ for a smooth manifold with boundary, giving $K_0(\cK_b(X))=0$ and $K_1(\cK_b(X))=\bZ^{1-p}$ with $p$ the number of boundary components, which has the non trivial consequence that any elliptic $b$-operator on a manifold with boundary can be endowed with Fredholm boundary conditions.
\end{itemize}

{\bf Computation of the K-theory groups in terms of corner cycles.} In this paper we explicitly compute the above K-theory groups for any finite product of manifolds with corners of codimension $\leq 3$ in terms corner cycles (explanation below). Our computations and results are based on a geometric interpretation of the algebras of $b$-pseudodifferential operators in terms of Lie groupoids. We explain and recall the basic facts on groupoids and the b-pseudodifferential calculus in the first two sections. Besides being extremely useful to compute $K-$theory groups, the groupoid approach we propose reveals to be very powerful to compute index morphisms and relate several indices. Indeed, the relation between the different indices for manifolds with corners was only partially understood for some examples. Let us explain this in detail. Let $X$ be a manifold with corners. Let $F_p=F_p(X)$ be the set of (without boundary, connected) faces of $X$ of codimension $p$. To compute $K_*(\cK_b(X))$, we use an increasing filtration of $X$ given by the open subspaces: 
\begin{equation}
    X_p = \bigcup_{k\le p\ ;\ f\in F_k} f.
\end{equation}
We have $X_0=\overset{\circ}{X}$ and $X_d=X$. We extend if necessary the filtration over $\ZZ$ by setting $X_k=\emptyset$ if $k<0$ and $X_k=X$ if $k>d$. The  $C^*$-algebra of $\cK_b(X)$ inherits (for entire details see section \ref{secK_b}) an increasing filtration by $C^*$-ideals:
\begin{equation}\label{eq:K-filt}
   \cK(L^2(\overset{\circ}{X}))=A_0\subset A_1 \subset \ldots\ A_d=A=\cK_b(X).
\end{equation}
The spectral sequence $(E^*_{*,*}(\cK_b(X)),d^*_{*,*})$ associated with this filtration can be used, in principle, to have a better understanding of these K-theory groups.  This filtration was also considered  by Melrose and Nistor in \cite{MelNis} and their main theorem is the expression of the first differential (theorem 9 ref.cit.). In trying to figure out an expression for the differentials of this spectral sequence in all degrees, we found a differential $\ZZ$-module $(\cC(X),\delta^\pcn)$ constructed in a very simple way out of the set of open connected faces of the given manifold with (embedded) corners $X$. Roughly speaking, the $\ZZ$-module $\cC(X)$ is generated by open connected faces provided with a co-orientation (that is, an orientation of their conormal bundles in $X$), while the differential map $\delta^\pcn$ associates to a given co-oriented face of codimension $p$, the sum of co-oriented faces of codimension $p-2k-1$, $k\ge 0$, containing it in their closures. This gives a well defined differential module for two reasons. The first one is that once a labelling of the boundary hyperfaces is chosen, the co-orientation of a given face induces co-orientations of the faces containing it in their closures, proving that the module map $\delta^\pcn$ is well defined. The second one is that the jumps by $2k+1$, $k\ge 0$, in the codimension guarantee the relation $\delta^\pcn\circ \delta^\pcn=0$. 
We call {\sl periodic conormal homology} and denote it by $H^{\pcn}(X)$ the homology of $(\cC(X),\delta^\pcn)$. Note that it is $\ZZ_2$-graded by sorting faces by even/odd codimension. 

Actually, it happens that the differential $\delta^\pcn$ retracts onto the simpler differential map $\delta$ where one stops at $-1$  in the codimension, that is,   $\delta$ maps a given co-oriented face of codimension $p$ to the sum of co-oriented faces of codimension $p-1$ containing it in their closures. We call {\sl conormal homology} and denote it by $H^{\cn}(X)$ the homology of $(\cC(X),\delta)$. Note that it is $\ZZ$-graded by sorting faces by codimension and that the resulting $\ZZ_2$-grading coincides with the periodic conormal groups. For full details about these homological facts see Sections \ref{seccnhom} and \ref{appendix}.

The conormal $\ZZ$-graded complex $(\cC_*(X),\delta)$ first appears in the work of Bunke \cite{Bunke} where it is used to compute obstructions for tamings of Dirac operators on manifolds with corners, and it also implicitely appears in the work of Melrose and Nistor in \cite{MelNis}, through the quasi-isomorphism that we prove here (Corollary \ref{cordiff1}). We can conclude this remark by recording that there is a natural isomorphism 
\begin{equation}
   H_{p}^{\cn}(X) \simeq E^2_{p,0}(\cK_b(X)).
\end{equation}
Our main K-theory computation can now be stated (theorem \ref{thmPCHvsKth}):

\begin{theorem}\label{thmPCHvsKthintro}
 Let $X=\Pi_iX_i$ be a finite product of manifolds with corners of codimension less or equal to three. There are natural isomorphisms 
\begin{equation}
\xymatrix{
H_{\mathrm0}^{\mathrm{pcn}}(X)\otimes_\mathbb{Z}\mathbb{Q}\ar[r]^-{\phi_X}_-\cong & K_0(\cK_b(X)))\otimes_\mathbb{Z}\mathbb{Q}
}
\text{ and }\qquad 
\xymatrix{
H_{\mathrm1}^{\mathrm{pcn}}(X)\otimes_\mathbb{Z}\mathbb{Q}\ar[r]^-{\phi_X}_-\cong & K_1(\cK_b(X))\otimes_\mathbb{Z}\mathbb{Q}.
}
\end{equation}
In the case $X$ contains a factor of codimension at most two or $X$ is of codimension three the result holds even without tensoring by $\mathbb{Q}$.
\end{theorem}


We insist on the fact that (periodic) conormal homology  groups are easily computable, for the underliyng chain complexes as well as the differentials maps are obtained from elementary and explicit ingredients. To continue let us introduce the Corner characters.
 
\begin{definition}[Corner characters]\label{cornercharactersintro}
Let $X$ be a manifold with corners. We define the even conormal character of $X$ as the finite sum
\begin{equation}
\chi_0(X)=dim_\mathbb{Q}\,H_{0}^{\pcn}(X)\otimes_\mathbb{Z}\mathbb{Q}.
\end{equation}
Similarly, we define the odd conormal character of $X$ as the finite sum
\begin{equation}
\chi_1(X)=dim_\mathbb{Q}\,H_{1}^{\pcn}(X)\otimes_\mathbb{Z}\mathbb{Q}.
\end{equation}
\end{definition}

We can consider as well
\begin{equation}
\chi(X)=\chi_0(X)-\chi_1(X),
\end{equation}	
then by definition
\begin{equation}
\chi(X)=1-\# F_1+\# F_2-\cdots +(-1)^d\#F_d
\end{equation}
We refer to the integer $\chi(X)$ as the Euler corner character of $X$.

In particular one can rewrite the theorem above to have, for $X$ as in the statement,

\begin{equation}\label{tableKthprodintro}
{\large 
\xymatrix{
&K_0(\cK_b(X))\otimes_\mathbb{Z}\mathbb{Q}\cong\mathbb{Q}^{\chi_0(X)}&\\
&K_1(\cK_b(X))\otimes_\mathbb{Z}\mathbb{Q}\cong\mathbb{Q}^{\chi_1(X)}&
}
}
\end{equation}
and, in terms of the corner character,
\begin{equation}
{\large \chi(X)=rank(K_0(\cK_b(X))\otimes_\mathbb{Z}\mathbb{Q})-rank(K_1(\cK_b(X))\otimes_\mathbb{Z}\mathbb{Q}).}
\end{equation}
Or in the case $X$ is a finite product of manifolds with corners of codimension at most 2  we even have
 \begin{equation}\label{tableKth}
K_0(\cK_b(X))\simeq \mathbb{Z}^{\chi_0(X)} \qquad \text{ and } \qquad 
  K_1(\cK_b(X))\simeq \mathbb{Z}^{\chi_1(X)}
\end{equation}
and also 
\(
 \chi_\cn(X)=rank(K_0(\cK_b(X)))-rank(K_1(\cK_b(X)))
\).

\vspace{1mm}

We can finally state the following primary obstruction Fredholm Perturbation theorem (theorem \ref{thmFPcornercycles}) in terms of corner's characters and corner's cycles.

\begin{theorem}\label{thmFPcornercyclesintro}
Let $X$ be a compact manifold with corners of codimension greater or equal to one. If $X$ is a finite product of manifolds with corners of codimension less or equal to three we have that 
\begin{enumerate}
\item If $X$ satisfies the Fredholm Perturbation property then the even Euler corner character of $X$ vanishes, {\it i.e. $\chi_0(X)=0$}.
\item If the even Periodic conormal homology group vanishes, {\it i.e. $H_0^{\pcn}(X)=0$} then $X$ satisfies the stably homotopic Fredholm Perturbation property.
\item If $H_0^{\pcn}(X)$ is torsion free and if the even Euler corner character of $X$ vanishes, {\it i.e. $\chi_0(X)=0$} then $X$ satisfies the stably homotopic Fredholm Perturbation property.
\end{enumerate}
\end{theorem}


We believe that the results above hold beyond the case of finite products of manifolds with corners of codimension $\leq 3$. On one side conormal homology can be defined and computed in all generality and in all examples we have the isomorphisms above still hold. The problem in general is to compute beyond the third  differential of the naturally associated spectral sequence for the K-theory groups for manifolds with corners of codimension greater or equal to four. This is technically a very hard task and besides explicit interesting examples become rare (not products). In fact, for any codimension, the correspondant spectral sequence in periodic conormal homology collapses at the second page as shown in the appendix. We believe it does collapse as well for K-theory because the results above. Another problem is related with the possible torsion of the conormal homology groups, indeed we prove in theorem \ref{freePCN} that for a finite product of manifolds with corners of codimension at most two these groups are torsion free and that the odd group for a three codimensional manifold with corners is torsion free as well. We think that in general these groups are torsion free but the combinatorics become very hard and one needs a good way to deal with all these data. We will discuss all these topics elsewhere.

Partial results in the direction of this paper were enterprised by several authors, we have already mentioned the seminal works of Melrose and Nistor in \cite{MelNis} and of Nazaikinskii, Savin and Sternin in \cite{NSS1} and \cite{NSS2}. In particular Melrose and Nistor start the computation of the $K-$groups of the algebra of zero order $b$-operators and some particular cases of Boundary analytic index morphisms as defined here (together with some topological formulas for them). Also, Nistor solves in \cite{NisGauge} the Fredholm Perturbation problem for manifolds with corners of the form a canonical simplex times a smooth manifold. Let us mention also the work of Monthubert and Nistor, \cite{MontNis}, in which they construct a classifying space associated to a manifold with corners whose $K-$theory can be in principle used to compute the analytic index above. We were very much inspired by all these works. In a slightly different framework,  Bunke  \cite{Bunke} studies  the obstruction  for the existence of {\sl tamings} of Dirac operators (that is, perturbations to  invertible ones) on manifolds with corners of arbitrary codimension, and also expresses these obstructions in terms of complexes associated with the faces. He then studies  local index theory and analytic obstruction theory for families.

The theorems above show the importance and interest of computing the Boundary Analytic and the Fredholm indices associated to a manifold with corners and if possible in a unified and in a topological/geometrical way. Using $K$-theory as above, for the case of a smooth compact connected manifold, the computation we are mentioning is nothing else that the Atiyah-Singer index theorem, \cite{AS}. As we mentioned already, for manifolds with boundary, Atiyah-Patodi-Singer gave a formula for the Fredholm index of a Dirac type operator. In fact, with the groupoid approach to index theory, several authors have contributed to the now realizable idea that one can actually use these tools to have a nice $K-$theoretical framework and to actually compute more general index theorems as in the classic smooth case. For example, in our common work with Monthubert, \cite{CLM}, we give a topological formula for the Fredholm Index morphism for manifolds with boundary that will allow us in a sequel paper to compare with the APS formula and obtain geometric information on the eta invariant. In the second paper of this series we will generalize our results of \cite{CLM} for general manifolds with corners by giving explicit topological index computations for the indices appearing above. 

\section{Melrose b-calculus for manifolds with corners via groupoids}

\subsection{Preliminaries on groupoids, K-theory $C^*$-algebras and Pseudodifferential Calculus}  

All the material in this section is well known and by now classic for the people working in groupoid's $C^*$-algebras, $K$-theory and index theory. For more details and references the reader is sent to \cite{DL10}, \cite{NWX}, \cite{MP}, \cite{LMV}, \cite{HS83}, \cite{Ren}, \cite{AnRen}.

{\bf Groupoids:} Let us start with the definition.

\begin{definition}
A $\it{groupoid}$ consists of the following data:
two sets $\gr$ and $\go$, and maps
\begin{itemize}
\item[(1)]  $s,r:\gr \rightarrow \go$ 
called the source and range (or target) map respectively,
\item[(2)]  $m:\gr^{(2)}\rightarrow \gr$ called the product map 
(where $\gr^{(2)}=\{ (\gamma,\eta)\in \gr \times \gr : s(\gamma)=r(\eta)\}$),
\end{itemize}
such that there exist two maps, $u:\go \rightarrow \gr$ (the unit map) and 
$i:\gr \rightarrow \gr$ (the inverse map),
which, if we denote $m(\gamma,\eta)=\gamma \cdot \eta$, $u(x)=x$ and 
$i(\gamma)=\gamma^{-1}$, satisfy the following properties: 
\begin{itemize}
\item[(i).]$r(\gamma \cdot \eta) =r(\gamma)$ and $s(\gamma \cdot \eta) =s(\eta)$.
\item[(ii).]$\gamma \cdot (\eta \cdot \delta)=(\gamma \cdot \eta )\cdot \delta$, 
$\forall \gamma,\eta,\delta \in \gr$ when this is possible.
\item[(iii).]$\gamma \cdot x = \gamma$ and $x\cdot \eta =\eta$, $\forall
  \gamma,\eta \in \gr$ with $s(\gamma)=x$ and $r(\eta)=x$.
\item[(iv).]$\gamma \cdot \gamma^{-1} =u(r(\gamma))$ and 
$\gamma^{-1} \cdot \gamma =u(s(\gamma))$, $\forall \gamma \in \gr$.
\end{itemize}
Generally, we denote a groupoid by $\gr \rightrightarrows \go $. A morphism $f$ from
a  groupoid   $\hr \rightrightarrows \ho $  to a groupoid   $\gr \rightrightarrows \go $ is  given
by a map $f$ from $\gr$ to $\hr$ which preserves the groupoid structure, i.e.  $f$ commutes with the source, target, unit, inverse  maps, and respects the groupoid product  in the sense that $f(h_1\cdot h_2) = f (h_1) \cdot f(h_2)$ for any $(h_1, h_2) \in \hr^{(2)}$.

\end{definition}

For $A,B$ subsets of $\go$ we use the notation
$\gr_{A}^{B}$ for the subset 
\[
\{ \gamma \in \gr : s(\gamma) \in A,\, 
r(\gamma)\in B\} .
\]

We will also need the following definition:

\begin{definition}[Saturated subgroupoids]\label{defsaturated}
Let $\gr\rightrightarrows M$ be a groupoid.
\begin{enumerate}
\item A subset $A\subset M$ of the units is said to be saturated by $\gr$ (or only saturated if the context is clear enough) if it is union of orbits of $\gr$.
\item A subgroupoid 
\begin{equation}
\xymatrix{
\gr_1 \ar@<.5ex>[d]_{r\ } \ar@<-.5ex>[d]^{\ s}&\subset&\gr \ar@<.5ex>[d]_{r\ } \ar@<-.5ex>[d]^{\ s}  \\
M_1&\subset&M
}
\end{equation}
is a saturated subgroupoid if its set of units $M_1\subset M$ is saturated by $\gr$.

\end{enumerate}
\end{definition}

A groupoid can be endowed with a structure of topological space, or
manifold, for instance. In the case when $\gr$ and $\go$ are smooth
manifolds, and $s,r,m,u$ are smooth maps (with $s$ and $r$
submmersions), then $\gr$ is called a Lie groupoid. In the case of manifolds
with boundary, or with corners, this notion can be generalized to that
of continuous families groupoids (see \cite{Pat}) or as Lie groupoids if one considers the category of smooth manifolds with corners.

{\bf $C^*$-algebras:} To any Lie groupoid $\gr\rightrightarrows \go$ one has several $C^*-$algebra completions for the *-convolution algebra $C_c^\infty(\gr)$. Since in this paper all the groupoids considered are amenable we will be denoting by $C^*(\gr)$ the maximal and hence reduced $C^*$-algebra of $\gr$. From now on, all the groupoids are then going to be assumed amenable.

In the sequel we will use the following two results which hold in the generality of locally compact groupoids equipped with Haar systems.

\begin{enumerate}
\item Let $\gr_1$ and $\gr_2$ be two locally compact groupoids equipped with Haar systems. Then for locally compact groupoid $\gr_1\times \gr_2$ we have 
\begin{equation}
C^*(\gr_1\times \gr_2)\cong C^*(\gr_1)\otimes C^*(\gr_2).
\end{equation}
\item Let $\gr\rightrightarrows \go$ a locally compact groupoid with Haar system $\mu$. Let $U\subset \go$ be a saturated open subset, then $F:=\go\setminus U$ is a closed saturated subset. The Haar system $\mu$ can be restricted to the restriction groupoids $\gr_U:=\gr_U^U\rightrightarrows U$ and $\gr_F:=\gr_F^F\rightrightarrows F$, and we have the following short exact sequence of $C^*$-algebras:
\begin{equation}\label{suiteres}
\xymatrix{
0\ar[r]&C^*(\gr_U)\ar[r]^i&C^*(\gr)\ar[r]^r&C^*(\gr_F)\ar[r]&0
}
\end{equation}
where $i:C_c(\gr_U)\to C_c(\gr)$ is the extension of functions by zero and $r:C_c(\gr)\to C_c(\gr_F)$ is the restriction of functions.
\end{enumerate}

{\bf $K$-theory:} We will be considering the $K$-theory groups of the $C^*$-algebra of a groupoid, for space purposes we will be denoting these groups by
\begin{equation}
K^*(\gr):=K_*(C^*(\gr)).
\end{equation}

We will use the classic properties of the $K$-theory functor, mainly its homotopy invariance and the six term exact sequence associated to a short exact sequence. Whenever the groupoid in question is a space (unit's groupoid) $X$ we will use the notation 
\begin{equation}
K^*_{top}(X):=K_*(C_0(X)).
\end{equation}
to remark that in this case this group is indeed isomorphic to the topological $K-$theory group.

{\bf $\Psi$DO Calculus for groupoids.} A pseudodifferential operator on a Lie groupoid (or more generally a continuous family groupoid) $\gr$ is a family of peudodifferential
operators on the fibers of $\gr$ (which are smooth manifolds without
boundary), the family being equivariant under the natural action of $\gr$. 

Compactly supported pseudodifferential operators  form an algebra, denoted by
$\Psi^\infty(\gr)$. The algebra of order 0 pseudodifferential operators
can be completed into a $C^*$-algebra, $\overline{\Psi^0}(\gr)$. There
exists a symbol map, $\sigma$, whose kernel is $C^*(\gr)$. This gives
rise to the following exact sequence:
$$0 \to C^*(\gr) \to \overline{\Psi^0}(\gr) \to C_0(S^*(\gr))\to 0$$
where $S^*(\gr)$ is the cosphere bundle of the Lie algebroid of $\gr$.

In the general context of index theory on groupoids, there is an
analytic index which can be defined in two ways. The first way,
classical, is to consider the boundary map of the 6-terms exact
sequence in $K$-theory induced by the short exact sequence above:
$$ind_a: K_1(C_0(S^*(\gr))) \to K_0(C^*(\gr)).$$

Actually, an alternative is to define it through the tangent groupoid
of Connes, which was originally defined for the groupoid of a smooth
manifold and later extended to the case of continuous family groupoids
(\cite{MP,LMNpdo}). The tangent groupoid of a Lie groupoid $\gr\rightrightarrows\go$ is a Lie groupoid
$$\gr^{tan}=A(\gr)\bigsqcup \gr\times (0,1]\rightrightarrows \gr^{(0)}\times
[0,1],$$ 
with smooth structure given by the deformation to the normal cone construction, see for example \cite{Ca2} for a survey of this construction related with the tangent groupoid construction. 

Using the evaluation maps, one has two $K$-theory morphisms, $e_0: K_0(C^*(\gr^{tan}))
\to K^0_{top}(A\gr)$ which is an isomorphism (since $K_*(C^*(\gr\times (0,1]))=0$), and $e_1: K_*(C^0(\gr^{tan})) \to K_0(C^*(\gr))$. The analytic
index can be defined as
$$ind_a=e_1 \circ e_0^{-1}:K^0_{top}(A^*\gr) \to K_0(C^*(\gr)).$$
modulo the surjection $K_1(C_0(S^*(\gr))\to K^0(A^*\gr)$.

See \cite{MP, NWX, Mont, LMNpdo,VassoutJFA} for a
detailed presentation of pseudodifferential calculus on groupoids.

\subsection{Melrose b-calculus for manifolds with corners via the b-groupoid}

We start by defining the manifolds with corners we will be using in the entire paper.

A manifold with corners is a Hausdorff space covered by compatible coordinate charts with coordinate functions modeled in the spaces
$$\bR_k^n:=[0,+\infty)^k\times \bR^{n-k}$$
for fixed $n$ and possibly variable $k$.

\begin{definition}
A manifold with embedded corners $X$ is a Hausdorff topological space endowed with a subalgebra $C^{\infty}(X)\in C^0(X)$ satisfying the following conditions:
\begin{enumerate}
\item there is a smooth manifold $\tilde{X}$ and a map $\iota:X\to \tilde{X}$ such that $$\iota^*(C^\infty(\tilde{X}))=C^\infty(X),$$
\item there is a finite family of functions $\rho_i\in C^\infty(\tilde{X})$, called the defining functions of the hyperfaces, such that
$$\iota(X)=\bigcap_{i\in I}\{\rho_i\geq0\}.$$
\item for any $J\subset I$, 
\begin{center}
$d_x\rho_i(x)$ are linearly independent in $T^*_x\tilde{X}$ for all $x\in F_J:=\bigcap_{i\in J}\{\rho_i=0\}$.
\end{center}
\end{enumerate}
\end{definition}

{\bf Terminology:} In this paper we will only be considering manifolds with embedded corners. We will refer to them simply as manifolds with corners. We will also assume our manifolds to be connected. More general manifold with corners deserve attention but as we will see in further papers it will be more simple to consider them as stratified pseudomanifolds and desingularize them as manifolds with embedded corners with an iterated fibration structure.

\vspace{2mm}

Given a compact manifold corners $X$, Melrose\footnote{for entire details in the case with corners see the paper of Melrose and Piazza \cite{MelPia}} constructed in \cite{Mel} the algebra $\Psi_b^*(X)$ of $b$-pseudodifferential operators. The elements in this algebra are called $b-$pseudodifferential operators, the subscript $b$ identifies these operators as obtained by "microlocalization" of the Lie algebra of $C^\infty$ vector fields on $X$ tangent to the boundary. This Lie algebra of vector fields can be explicitly obtained as sections of the so called $b$-tangent bunlde $^bTX$ (compressed tangent bundle that we will appear below as the Lie algebroid of an explicit Lie groupoid).  The b-pseudodifferential calculus developed by Melrose has the classic and expected properties. In particular there is a principal symbol map
$$\sigma_b:\Psi_b^m(X)\to S^{[m]}(^bT^*X).$$  
Ellipticity has the usual meaning, namely invertibility of the principal symbol. Moreover (discussion below theorem 2.15 in \cite{MelPia}), an operator is elliptic if and only\footnote{Notice that this remark implies that to a $b$-pseudodifferential operator one can associate an "index" in the algebraic K-theory group $K_0(\Psi_b^{-\infty}(X))$ (classic construction of quasi-inverses).} if it has an quasi-inverse modulo $\Psi_b^{-\infty}(X)$. Now, the operators in $\Psi_b^{-\infty}(X)$ are not all compact (unless the topological boundary $\partial X=\emptyset$) but they contain a subalgebra consisting of compact operators (those for which certain indicial map is zero, p.8 ref.cit.). Hence, among elliptic $b$-pseudodifferential operators one has those for which the quasi-inverse is actually modulo compact operators and hence Fredholm (again, see p.8 ref.cit. for a characterization of Fredholm operators in terms of an indicial map), these $b$-elliptic operators are called fully elliptic operators.

Now, as every $0$-order $b$-pseudodifferential operator (ref.cit. (2.16)), the operators in $\Psi_b^{-\infty}(X)$ extend to bounded operators on $L^2(X)$ and hence if we consider its completion as bounded operators one obtains an algebra denoted in this paper by $\cK_b(X)$ that fits in a short exact sequence of $C^*-$algebras of the form
\begin{equation}\label{bcompact}
\xymatrix{
0\ar[r]&\cK(X)\ar[r]^-{i_0}&\cK_b(X)\ar[r]^-{r}&\cK_b(\partial X) \ar[r]&0
}
\end{equation}
where $\cK(X) $ is the algebra of compact operators in $L^2(X)$.

Let $X$ be a compact manifold with embedded corners, so by definition we are assuming there is a smooth compact manifold (of the same dimension) $\tilde{X} $ with $X\subset \tilde{X}$ and $\rho_1,..., \rho_n$ defining functions of the faces. In \cite{Mont}, Monthubert constructed a Lie groupoid (called Puff groupoid) associated to any decoupage $(\tilde{X}, (\rho_i))$, it has the following expression
\begin{equation}\label{Puffgrpd}
G(\tilde{X},(\rho_i))= \{(x,y,\lambda_1,...,\lambda_n)\in \tilde{X}\times \tilde{X}\times \RR^n: \rho_i(x)=e^{\lambda_i}\rho_i(y)\}.
\end{equation}
as a Lie subgroupoid of $\tilde{X}\times \tilde{X}\times \RR^k$. The Puff groupoid is not s-connected, denote by $G_c(\tilde{X},(\rho_i))$ its s-connected component.


\begin{definition}[The $b$-groupoid]\label{def:b-groupoid}
The $b-$groupoid $\Gamma_b(X)$ of $X$ is by definition the restriction to $X$ of the s-connected Puff groupoid (\ref{Puffgrpd}) considered above, that is
\begin{equation}
\Gamma_b(X):= G_c(\tilde{X},(\rho_i))|_X\rightrightarrows X
\end{equation}
\end{definition}

The $b-$groupoid was introduced by B. Monthubert in order to give a groupoid description for the Melrose's algebra of $b$-pseudodifferential operators. We summarize below the main properties we will be using about this groupoid:

\begin{theorem}[Monthubert \cite{Mont}]
Let $X$ be a manifold with corners as above, we have that
\begin{enumerate}
\item $\Gamma_b(X)$ is a $C^{0,\infty}$-amenable groupoid.
\item It has Lie algebroid $A(\Gamma_b(X))=^bTX$, the $b$-tangent bundle of Melrose.
\item Its $C^*-$algebra (reduced or maximal is the same since amenability) coincides with the algebra of $b$-compact operators. The canonical isomorphism
\begin{equation}
C^*(\Gamma_b(X))\cong \cK_b(X)
\end{equation}
is given as usual by the Schwartz Kernel theorem.
\item The pseudodifferential calculus of $\Gamma_b(X)$ coincides with compactly supported $b$-calculus of Melrose.
\end{enumerate}
\end{theorem}

\begin{remark}
To simplify, in the present paper, we only discuss the case of scalar operators. The case of operators acting on sections of vector bundles is treated as classically by considering bundles of homomorphisms. 
\end{remark}

\section{Boundary analytic and Fredholm Indices for manifolds with corners: relations and Fredholm Perturbation   characterization}

We will now introduce the several index morphisms we will be using, mainly the Analytic and the Fredholm index.
In all this section, $X$ denotes a compact and connected manifold with embedded corners. 

\subsection{Analytic and Boundary analytic Index morphisms}\label{subsectionAnalyticindexmorphism}

Any elliptic $b$-pseudodifferential $D$ has an analytical index $\mathrm{Ind}_{\mathrm{an}}(D)$ given by 
\[
\mathrm{Ind}_{\mathrm{an}}(D)=I([\sigma_b(D)]_1)\in K_0(\cK_b(X))
\]
where $I$ is the connecting homomorphism  in $K$-theory of exact sequence 
\begin{equation}\label{bKses}
\xymatrix{
0\ar[r]&\cK_b(X)\ar[r]&\overline{\Psi_b^0(X)}\ar[r]^-{\sigma_b}&C(^bS^*X)\ar[r]&0.
}
\end{equation}
and $[\sigma_b(D)]_1$ is the class in $K_1(C({}^bS^*X)$ of the principal symbol $\sigma_b(D)$ of $D$.

Alternatively, we can express $\mathrm{Ind}(D)$ using adiabatic deformation groupoid of $\Gamma_b(X)$ and the class in $K_0$ of the same symbol, namely:
\begin{equation}
 [\sigma_b(D)] = \delta([\sigma_b(D)]_1)\in K_0(C_0({}^bT^*X))
\end{equation}
where $\delta$ is the connecting homomorphism of the exact sequence relating the vector and  sphere bundles:
\begin{equation}\label{bTS}
\xymatrix{
0\ar[r]& C_0({}^bT^*X)\ar[r]& C_0({}^bB^*X)\ar[r]&C({}^bS^*X)\ar[r]&0.
}
\end{equation}
Indeed, consider the exact sequence 
\begin{equation}
\xymatrix{
0\ar[r]&C^*(\Gamma_b(X)\times (0,1])\ar[r]&C^*(\Gamma_b^{tan}(X))\ar[r]^-{r_0}&C^*(^bTX)\cong C_0(^bT^*X)\ar[r]&0,
}
\end{equation}
 in which the ideal is $K$-contractible and set 
\begin{equation}
Ind^a_X= r_1 \circ r_0^{-1} : K^0_{top}(^bT^*X)\longrightarrow K_0(\cK_b(X))
\end{equation}
where $r_1 : K_0(C^*(\Gamma_b^{tan}(X)))\to K_0(C^*(\Gamma_b(X))) $ is induced by the restriction morphism to $t=1$.
Applying a mapping cone argument to the exact sequence (\ref{bKses}) gives a commutative diagram
\begin{equation}
\xymatrix{
K_1(C(^bS^*X))\ar[rd]_-{\delta}\ar[rr]^-{I}&&K_0(\cK_b(X))\\
&K^0_{top}(^bT^*X)\ar[ru]_-{Ind^a_X}&
}
\end{equation}
Therefore we get, as announced:
\begin{equation}
  \mathrm{Ind}_{\mathrm{an}}(D) = Ind^a_X([\sigma_b(D)])
\end{equation}
 The map $Ind^a_X$ will be called  the {\it  Analytic Index morphism} of $X$.  A closely related homomorphism is the {\it  Boundary analytic Index morphism}, in which  the restriction to $X\times\{1\}$ is replaced  by the one to $\partial X\times\{1\}$, that is, we set:
 \begin{equation}
  Ind^\partial_X =  r_\partial \circ r_0^{-1}  : K_0(C_0(^bT^*X)\longrightarrow K_0(C^*(\Gamma_b(X)|_{\partial X})),
 \end{equation}
 where $r_\partial$ is induced by the   homomorphism  
 $C^*(\Gamma^{tan}_b(X))\longrightarrow C^*(\Gamma_b(X))|_{\partial X} $. We have of course 
 \begin{equation}
  Ind^\partial_X = r_{1,\partial}\circ Ind^a_X 
 \end{equation}
if $r_{1,\partial}$ denotes the map induced by the homomorphism  $C^*(\Gamma_b(X))\longrightarrow C^*(\Gamma_b(X)|_{\partial X})$.

\subsection{Fredholm Index morphism}\label{Fredsection}
In general, elliptic $b$-operators on $X$ are not Fredholm. Indeed, to  construct an inverse of a $b$-operator modulo compact terms,  we  have to invert not only the principal symbol but also all the family of boundary symbols. One way to summarize this situation is to introduce the  algebra of full, or joint, symbols. Let $\cH$ be the set of closed boundary hyperfaces of $X$, and set 
\begin{equation}
 \cA_{\cF}= \left\{ \Big(a,(q_H)_{H\in \cH}\Big) \in C^\infty({}^bS^*X)\times\prod_{H\in \cH}\Psi^0(\Gamma_b(X)\vert_H \quad ; \quad \forall H\in \cH, \ a\vert_H=\sigma_b(q_H) \right\}.
\end{equation}
The full symbol map: 
\begin{equation}
 \sigma_F : \Psi^0(\Gamma_b(X)\ni P \longmapsto \Big( \sigma_b(P),(P\vert_H)_{H\in \cH}\Big) \in \cA_{\cF}
\end{equation}
extends to the $C^*$-closures of the algebras and the assertion about the invertibility modulo compact operators amounts to the exactness of the sequence \cite{LMNpdo}:
\begin{equation}\label{Fredses}
\xymatrix{
0\ar[r]&\cK(X)\ar[r]&\overline{\Psi^0(\Gamma_b(X))}\ar[r]^-{\sigma_F}& \overline{\cA_{\cF}}\ar[r]&0
}
\end{equation}
Then one set:
\begin{definition}[Full Ellipticity]
An operator $D\in \Psi^0(\Gamma_b(X))$ is said to be fully elliptic if $\sigma_F(D)$ is invertible.
\end{definition}
We then recall the following result  of Loya \cite{Loya} (the statement also appears in \cite{MelPia}). Remember that $b$-Sobolev spaces $H^s_b(X)$ are defined using $b$-metrics and  $b$-operators map continuously   
$H_b^{m}(X)$ to $ H_b^{s-m}(X)$ for every $s$.
\begin{theorem}[\cite{Loya}, Theorem 2.3]\label{FELL=Fred}
An operator $D\in \Psi_b^0(X)$ is fully elliptic if and only if it is Fredholm on $H_b^{s}(X)$ for some $s$ (and then for any $s$, with Fredholm index independent of $s$).
\end{theorem}
For a given fully elliptic operator $D$, we denote by $\mathrm{Ind}_{\mathrm{Fred}}(D)$ its Fredholm index. 
We are going to express this number in terms of $K$-theory and clarify the relationship between the analytical index and  full ellipticity on $X$ using deformation groupoids. Let us star with the tangent groupoid
\begin{equation}
\Gamma_b(X)^{tan}:=(G_c(\tilde{X},(\rho_i))^{tan})|_{X\times [0,1]}=T_bX\bigsqcup \Gamma_b(X)\times (0,1]\rightrightarrows X\times [0,1].
\end{equation}
Now we introduce the two following saturated subspaces of $X\times[0,1]$:
\begin{equation}
 X_\cF:=X\times [0,1]\setminus \partial X\times \{1\} \text{ and } X_\partial:=X_\cF\setminus \intx \times (0,1]=X\cup \partial X\times[0,1).
\end{equation}
The {\bf Fredholm $b$-groupoid}  and the {\bf noncommutative tangent space} of $X$  are defined  by  
\begin{equation}\label{Fredgrpd}
\Gamma_b(X)^{Fred}:=\Gamma_b(X)^{tan}|_{X_\cF} \text{ and } T_{nc}X:=\Gamma_b(X)^{Fred}|_{X_\partial}
\end{equation}
respectively. They are obviously $KK$-equivalent as one sees using the  exact sequence:
\begin{equation}\label{btangentsuite}
\xymatrix{
0\ar[r]&C^*(\intx\times \intx\times (0,1])\ar[r]&C^*(\Gamma_b(X)^{\cF})\ar[r]^-{r_F}&C^*(T_{nc}X)\ar[r]&0
}
\end{equation}
whose ideal is $K$-contractible. We then define the  {\bf Fredholm index morphism} by:
\begin{equation}\label{Fredmorph}
Ind^X_F=(r_1)_*\circ(r_F)_{*}^{-1}:K^0(T_{nc}X)\longrightarrow
K^0(\intx\times \intx)\simeq \mathbb{Z},
\end{equation}
 Following \cite[Definition 10.4]{DLR}, we denote  by $\mathrm{FE}(X)$ the group of order $0$ fully elliptic operators modulo stable homotopy. Then the vocabulary above is justified by:
\begin{proposition}\label{FullyAPS}
There exists a group isomorphism
\begin{equation}
  \sigma_{nc} : \mathrm{FE}(X)  \longrightarrow K_0(C^*(\mathcal{T}_{nc}X))
\end{equation}
such that   
\begin{equation}
 r_0([\sigma_{nc}(D)])= [\sigma_b(D)]\in K_0(C_0({}^bT^*X))
\quad \text{ and } \quad
Ind_F^X([\sigma_{nc}(D)])= \mathrm{Ind}_{\mathrm{Fred}}(D),
\end{equation}
where $r_0$ comes from the  natural restriction map $C^*(\mathcal{T}_{nc}X)\to C_0({}^bT^*X)$. 
\end{proposition}
This is proved by the  method  leading  to  \cite[Theorem 4]{Savin2005} and \cite[Theorem 10.6]{DLR} exactly in the same way. Also, this homotopy classification appears in  \cite{NSS2}, in which the $K$-homology of a suitable {\sl dual manifold} is used  instead of the $K$-theory of the noncommutative tangent space.   
Previous related results appeared in \cite{LMNpdo} for differential operators and using different algebras to classify their symbols.

\noindent The construction of the various index maps above is summarized into the commutative diagram:
\begin{equation}\label{AnavsFredses}
\tiny{\xymatrix{
&K^0(\Gamma_b^F)\ar[rr]^-{i_F}\ar[ld]|-{e_1}&&K^0(\Gamma_b^{tan})\ar[rr]^-{r_\partial}\ar[ld]|-{e_1}&&K^0(\Gamma_b|_\partial)\ar[ld]|-{Id}\ar[dd]&\\
K^0(\intx\times\intx)\ar[rr]^-{i_0}&&K^0(\Gamma_b)\ar[rr]^-{r_b}&&K^0(\Gamma_b|_\partial)\ar[dd]&&\\
&K^1(\Gamma_b|_\partial)\ar@{.>}[uu]&&K^1(\Gamma_b^{tan})\ar@{.>}[ld]|-{e_1}\ar@{.>}[ll]^-{r_\partial}&&K^1(\Gamma_b^F)\ar@{.>}[ll]\ar[ld]|-{e_1}&\\
K^1(\Gamma_b|_\partial)\ar[uu]^-{\partial_1}\ar@{<.}[ru]|-{Id}&&K^1(\Gamma_b)\ar[ll]^-{r_b}&&K^1(\intx\times\intx)\ar[ll]^-{i_0}&&
}}
\end{equation}

\subsection{Fredholm perturbation property}

We are ready to define the Fredholm Perturbation Property \cite{NisGauge} and its stably homotopic version.
\begin{definition} Let $D\in \Psi_b^m(X)$ be elliptic. We say that $D$ satisfies: 
\begin{itemize}
\item  the {\it Fredholm Perturbation Property} $(\cF\cP)$ if there is   $R\in \Psi_b^{-\infty}(X)$ such that $D+R$ is  fully elliptic.  
\item the {\it stably homotopic Fredholm Perturbation Property} $(\cH\cF\cP)$ if there is a fully elliptic operator $D'$ with $[\sigma_b(D')]=[\sigma_b(D)]\in K_0(C^*({}^bTX))$.
\end{itemize}
\end{definition}
We also say that $X$ satisfies the   {\it (stably homotopic) Fredholm Perturbation Property}  if any elliptic $b$-operator on $X$ satisfies $((\cH)\cF\cP)$.

Property  $(\cF\cP)$ is of course stronger than property $(\cH\cF\cP)$.  In \cite{NisGauge}, Nistor characterized $(\cF\cP)$ in terms of the vanishing of an index in some particular cases. In  \cite{NSS2}, Nazaikinskii, Savin and Sternin characterized $(\cH\cF\cP)$
 for arbitrary manifolds with corners using an index map associated with their dual manifold construction. We now rephrase the result of \cite{NSS2} in terms of deformation groupoids. 
\begin{theorem}\label{AnavsFredthm1}
Let $D$ be an elliptic $b$-pseudodifferential operator on a compact manifold with corners $X$. Then  $D$ satisfies $(\cH\cF\cP)$ if and only if 
\(
 Ind_\partial([\sigma_b(D)])=0
\).\\
In particular if $D$ satisfies $(\cF\cP)$ then its analytic indicial index vanishes. 
\end{theorem}
\begin{proof}
Note that the Fredholm and the tangent groupoids are related by the exact sequence
\begin{equation}\label{ex-seq:Fred-tan-partialb}
\xymatrix{
0\ar[r]&C^*(\Gamma_b^{Fred}(X))\ar[r]^-{i_F}&C^*(\Gamma_b^{tan}(X))\ar[r]^-{r_\partial}&C^*(\Gamma_b(X)_{\partial X})\ar[r]&0
}
\end{equation}
Then Proposition \ref{FullyAPS}, together with this  exact sequence and the commutative diagram: 
 \begin{equation}
 \xymatrix{
 K_0(C^*(\Gamma_b^F)) \ar[d]_{i_F}\ar[r]^{\simeq}_{r_F} & K_0(C^*(T_{nc}X)) \ar[d]_{r_0} \\
  K_0(C^*(\Gamma_b^{tan}))  \ar[r]^{\simeq}_{r_0} & K_0(C^*({}^bTX)) 
 } 
\end{equation}
 yields the result. 
\end{proof}

Loosely speaking, this theorem tells us that the $K$-theory of $\Gamma_b(X)_{\partial X}$, or equivalently the one of $\Gamma_b(X)$ as we shall see later, is the receptacle for the obstruction to Fredholmness of elliptic symbols in the $b$-calculus. This is why we now focus on the understanding of these $K$-theory groups. If the result is well kwown in codimension less or equal to $1$, the general case is far from understood.  Meanwhile, we will also clarify the equivalent role played by $\Gamma_b(X)$ and  $\Gamma_b(X)_{\partial X}$.

\section{The conormal homology of a manifold with corners}\label{seccnhom}
 
 In all this section,  $X$ is a manifold with embedded corners of codimension $d$, whose connected hyperfaces $H_1,\ldots,\ldots H_N$ are provided with defining functions $r_1,\ldots,r_N$.
 
\subsection{Definition of the homology}\label{subsec:def-cn}
The one form  $e_j=dr_j$ trivialises  the conormal bundle of $H_j$ for any $1\le j\le N$.   By convention,  $p$-uples of integers $I=(i_1,\ldots,i_p)\in \NN^p$ are always labelled so that $1\le i_1< \ldots<i_p \le N$. Let  $I$ be a  $p$-uple, set 
\begin{equation}
 H_I = r^{-1}_I(\{0\}) = H_{i_1}\cap\ldots\cap H_{i_p}.
\end{equation}
and note $c(I)$ the set of open connected faces of codimension $p$ included in $H_I$. Also, we denote by $e_I$ the exterior product 
\begin{equation}
  e_I = e_{i_1}. e_{i_2}.\ldots . e_{i_p}.
\end{equation}
Let  $f$ be a face of codimension $p$ and $I$ the  $p$-uple such that $f\in c(I)$. The  conormal bundle $N(f)$ of  $f$ has a global basis given by the sections $e_j$, $j\in I$, and its  orientations are identified with $\pm e_I$.
For any integer $0\le p\le d $, we denote by $C_p(X)$ the free $\ZZ$-module  generated by 
\begin{equation}
 \{ f\otimes \varepsilon \ ;\ f\in F_p,\ \varepsilon \text{ is an orientation of } N(f)  \}.
\end{equation}
 Let $f\in F_p$,  $\epsilon_f$ an orientation of  $N(f)$ and  $g\in F_{p-1}$  such that $f\subset \overline{g}$.  The face $f$ is  characterized in $\overline{g}$ by the vanishing of  a defining function  $r_{i_{(g,f)}}$. Then  the contraction  $ e_{i_{(g,f)}}\lrcorner \epsilon_f$ is an orientation of $N(g)$. Recall that the contraction $\lrcorner$ is defined by 
\begin{equation}
 e_i\lrcorner e_I = 
 \begin{cases}
      0 & \text{ if } i\not\in I \\ 
      (-1)^{j-1} e_{I\setminus\{i\}} & \text{ if } i \text{ is the }  j^{\text{th}} \text{ coordinate  of } I .
  \end{cases}
 \end{equation} 
We then define  $\delta_p : C_p(X)\to C_{p-1}(X)$ by 
\begin{equation}\label{diffpcn1}
   \delta_p(f\otimes \varepsilon_f) = \sum_{\substack{g\in F_{p-1},  \\ f\subset\overline{g}}} g\otimes e_{i_{(g,f)}}\lrcorner\varepsilon_f.
\end{equation}
It is not hard to check directly that $(C_*(X),\delta_*)$ is a differential complex. Actually, $\delta_*$  is the  component of  degree $-1$  of another natural differential map $\delta^{\pcn}=\sum_{k\ge 0}\delta^{2k+1}$, which eventually produces a quasi-isomorphic differential complex. Details are provided in Section \ref{appendix}. 

We define the {\it conormal homology} of $X$ as the homology of $(C_*(X),\delta_*)$, and we note 
\begin{equation}
  H^{\cn}_p(X) := H_p(C_*(X),\delta_*).
 \end{equation}
This homology was first considered in \cite{Bunke}, in a slightly different but equivalent way. 
Also, the graduation of the conormal homology into even and odd degree, called here {\it periodic conormal homology},  will be used and we note 
 \begin{equation}
 H^{\pcn}_{0}(X)=\oplus_{p\ge 0} H^{\cn}_{2p}(X) \text{ and }H^{\pcn}_{1}(X)=\oplus_{p\ge 0} H^{\cn}_{2p+1}(X).
\end{equation}

\subsection{Examples}\label{exples-cn}
The determination of  the  groups $ H^{\cn}_*(X)$ is completely elementary in all concrete cases. In the following examples, it is understood that faces $f$ arise with the orientation given by $e_I$ if $f\in c(I)$.
\begin{example}\ 
\begin{itemize}
\item Assume that $X$ has no boundary. Then   $H^{\pcn}_{0}(X)=H^{\cn}_{0}(X)\simeq\ZZ$, $H^{\pcn}_{1}(X)=0$.

\item Assume that $X$ has a boundary with $n$ connected components. Then $H^{\pcn}_{0}(X)=0$ and  $H^{\pcn}_{1}(X)=H^{\cn}_{1}(X)\simeq \ZZ^{n-1}$. More precisely, if  we set $F_1=\{f_1,...,f_n\}$ then $\{f_1-f_2,f_2-f_3,...,f_{n-1}-f_n\}$ provides a basis of $\ker \delta_1$.
  
\item Assume that $X$ has codimension $2$ and that $\partial X$ is connected. Then $H^{\pcn}_{0}(X)=H^{\cn}_{2}(X)=\ker \delta_2\simeq\ZZ^k$, where all nonnegative integers $k$ can arise. For instance, consider the unit closed ball $B$ in $\RR^3$, cut $k+1$ small disjoint disks out of its boundary and glue two copies of such spaces  along the pairs of cut out disks. We get a space $X$ satisfying the statement: the boundaries $s_0,\ldots, s_{k}$ of the original disks provide a basis of $F_2$ and the family $s_0-s_j$, $1\le j\le k$ a basis of $\ker \delta_2$.  Finally, $([0,+\infty))^2$ provides an example with $k=0$.

\item Consider the cube  $X=[0,1]^3$. 
\begin{enumerate}
 \item We have $H^{\pcn}_{0}(X)=0$ and  $H^{\pcn}_{1}(X)=H^{\cn}_3(X)\simeq \ZZ$.
 \item Remove a small open cube into the interior of $X$ and call the new space $Y$. Then 
 \[
   H^{\pcn}_{0}(Y)=0 \text{ and } H^{\pcn}_{1}(Y)= H^{\cn}_{3}(Y)\oplus  H^{\cn}_{1}(Y) \simeq \ZZ^2\oplus \ZZ.
 \]
  \item Remove a small open ball into the interior of $X$ and call the new space $Z$. Then 
  \[
   H^{\pcn}_{0}(Z)=0 \text{ and } H^{\pcn}_{1}(Y)= H^{\cn}_{3}(Y)\oplus  H^{\cn}_{1}(Y) \simeq \ZZ\oplus \ZZ.
 \]
\end{enumerate}
 \end{itemize}
\end{example}

\subsection{Long exact sequence in conormal homology}
 
We define a filtration of   $X$ by  open submanifolds with corners by setting:
\begin{equation}\label{filtration-by-codimension}
   X_m = \bigcup_{f\in F_k, \ k\le m}f, \qquad 0\le m\le d. 
\end{equation}
This leads to differential complexes  $(C_*(X_m),\delta)$ for $0\le m\le d$.
We can also filtrate the  differential complex $(C_*(X),\delta)$ by the codimension of faces: 
\begin{equation}
 F_m(C_*(X))= \bigoplus_{k=0}^m C_k(X).
\end{equation}
There is an obvious identification $C_*(X_m)\simeq F_m(C_*(X))$ and we thus consider  $(C_*(X_m),\delta)$ as a subcomplex of $(C_*(X),\delta)$, with  quotient complex denoted by $(C_*(X,X_m),\delta)$. The quotient module is also naturally embedded in $C_*(X)$:
\begin{equation}
 C_*(X,X_m) = C_*(X)/C_*(X_m)\simeq  \bigoplus_{k=m+1}^d C_k(X)\subset C_*(X).
\end{equation}
The embedding,  denoted by $\rho$, is a section of the quotient map. 
The short exact sequence:  
\begin{equation}\label{short-seq-cn}
 \xymatrix{
 0 \ar[r] & C_*(X_m)\ar[r]  &  C_*(X) \ar[r]  &  C_*(X,X_m)\ar[r]  & 0
 }
\end{equation}
induces a long exact sequence in conormal homology:
\begin{equation}\label{long-seq-cn}
{\small
\xymatrix{
\cdots\ar[r]^-{\partial_{p+1}}&H_{p}^{\cn}(X_m)\ar[r]&H_{p}^{\cn}(X)\ar[r]&H_{p}^{\cn}(X, X_m)\ar[r]^-{\partial_p}& H_{p-1}^{\cn}(X_m)\ar[r]& \cdots
}}
\end{equation}
and we need to precise the connecting homomorphism. 
\begin{proposition}\label{cnhomologyconnectingmap}
Let $[c]\in H_p^{\cn}(X,X_m)$. Then 
\begin{equation}
 \partial_p[c]= [\delta(\rho(c))].
\end{equation}
\end{proposition}
\begin{proof}
 Since $c$ is by assumption a cycle in $(C_*(X,X_m),\delta)$, the chain $\rho(c)$ has a boundary made of faces contained in $X_m$. The result follows. 
\end{proof}
\begin{remarks}\ \\
$\bullet$ We can replace $X$ by $X_l$ and quotient the exact sequence \eqref{short-seq-cn} by $C_*(X_q)$ for any integers $l,m,q$ such that $0\le q\le m\le l\le d$.  This leads to long exact sequences:
  \begin{equation}\label{long-seq-cn-variant}
{\small
\xymatrix{
\cdots\ar[r]^-{\partial}&H_{p}^{\cn}(X_{m},X_q)\ar[r]&H_{p}^{\cn}(X_l,X_q)\ar[r]&H_{p}^{\pcn}(X_l, X_{m})\ar[r]^-{\partial}& H_{p-1}^{\pcn}(X_{m},X_q)\ar[r]& \cdots
}}
\end{equation}
whose connecting homomorphisms are  again  given by the  formula of Proposition \ref{cnhomologyconnectingmap}. 

\noindent $\bullet$ If we split  the conormal homology into even and odd periodic groups, then the long exact sequence \eqref{long-seq-cn} becomes a six term exact sequence:
\begin{equation}\label{six-seq-cn}
\xymatrix{
  H^{\pcn}_0(X_m)\ar[r]&   H^{\pcn}_0(X) \ar[r] &  H^{\pcn}_0(X,X_m) \ar[d]^-{\partial^0} \\
   H^{\pcn}_1(X,X_m)\ar[u]^{\partial^1} & H^{\pcn}_1(X) \ar[l]& H^{\pcn}_1(X_m)\ar[l].
}
\end{equation}
where $\partial^0,\partial^1$ are given by the direct sum in even/odd degrees of the maps $\partial_*$  of Proposition \ref{cnhomologyconnectingmap}. 

\noindent $\bullet$ We can replace $X_m$ in  the exact sequence \eqref{short-seq-cn}  by an open saturated submanifold $U\subset X_m$, that is, an open subset of $X$  consisting of an union of faces.  This gives in the same way a subcomplex $(C_*(U),\delta)$ of $(C_*(X),\delta)$ and a section  $\rho : C_*(X,U) \to C_*(X)$ allowing to state Proposition \ref{cnhomologyconnectingmap} verbatim.  More generally, if $U$ is any open submanifold of $X$ and $\widetilde{U}$ denotes the smallest open saturated submanifold containing $U$, then any face $f$ of $U$ is contained in a unique face $\widetilde{f}$ of $X$ and an oriention of $N(f)$ determines an orientation of $N(\widetilde{f})$. This gives rise to a quasi-isomorphism $C_*(U)\to C_*(\widetilde{U})$.
\end{remarks}
Finally, assume that $d\ge 1$. Since $X$ is connected, the map $\delta_1:C_1(X)\to C_0(X)$ is surjective, which implies by Proposition \ref{cnhomologyconnectingmap} the surjectivity of the connecting homomorphism $\partial^1 : H_1^{\pcn}(X,X_0)\to H_0^{\pcn}(X_0)$. This fact and $H_1^{\pcn}(X_0)=0$ gives, using \eqref{six-seq-cn},  the useful corollary:
\begin{corollary}\label{corAdA0isoPCN}
For any connected manifold with corners $X$ of codimension $d\geq 1$ the canonical morphism $H_0^{\pcn}(X)\to H_0^{\pcn}(X,X_0)$ is an isomorphism.
\end{corollary}

\subsection{Torsion free in low codimensions}
Here we will show that up to codimension 2 the conormal homology groups (and later on the K-theory groups) are free abelian groups.

\begin{lemma}\label{lem:Hcn1}
  Let $X$ be of arbitrary codimension and assume that $\partial X$ has $l$ connected components. Then $H^{\cn}_{1}(X)\simeq\ZZ^{l-1}$.
\end{lemma}
\begin{proof}
  For any face $f$, denote by $cc(f)$ the connected component of  $\partial X$ containing $f$. It is obvious that $\ker \delta_1$ is generated by the differences $f-g$ where $f,g$ run through $F_1$. Let $f,g\in F_1$ such that $cc(f)=cc(g)$. Then  there exist $f_{0},\ldots,f_{l}\in F_1$  such that $f=f_0$, $g=f_l$ and  $\overline{f_{i}}\cap \overline{f_{i+1}}\not=\emptyset$ for any $i$. Therefore for any $i$, there exists $ f_{i,i+1}\in F_2$ such that  $\delta_2( f_{i,i+1})=f_{i}-f_{i+1}$, hence $f-g = \delta_2(\sum   f_{i,i+1})$ is a boundary in conormal homology. 

Now assume that $cc(f)\not=cc(g)$.  By the previous discussion, we also have $[f-g] =[f'-g']\in H^{\cn}_1(X)$ for any $f',g'\in F_1$ such that  $f'\subset cc(f)$ and $g'\subset cc(g)$. Therefore, pick up one hyperface in each connected component of $\partial X$, call them $f_1,\ldots,f_l$, and set $\alpha_i=[f_1-f_i]\in H^{\cn}_1(X)$ for $i\in\{2,\ldots, l\}$. It is obvious that $(\alpha_i)_{2\le i\le l}$ generates $ H^{\cn}_1(X)$. So, consider integers $x_2,\ldots,x_l$ such that 
\[
 \sum_{i=2}^l x_i\alpha_i = 0.
\]
In other words, there exists $x\in C_2(X)$ such that 
\begin{equation}\label{free-basis}
   (\sum_{i=2}^l x_i)f_1 - \sum_{i=2}^l x_if_i = \delta_2(x).
\end{equation}
For any $p\ge 1$ and $2\le j\le l$ denote by $\pi_{j}: C_{p}(X)\to C_{p}(X)$ the map defined by $\pi_j(h)=h$ if $h\subset cc(f_j)$ and  $\pi_j(h)=0$ otherwise. All $\pi_{i}$s commute with $\delta_*$, hence \eqref{free-basis} gives
\[
 \forall 2\le j\le l, \quad x_j f_j = \delta_2(\pi_j(x)). 
\]
Since $\delta_1(f_j)= \overset{\circ}{X}\not=0$, we conclude $x_j=0$ for all $j$.  
\end{proof}

\begin{theorem}\label{freePCN}
Let assume that  $X$ is connected and has codimension $d\le 2$. Then $H_*^{\pcn}(X)$ is a free abelian group. 
\end{theorem}
\begin{proof}
This is essentially a compilation of previous examples and computations. The first two cases in Example \ref{exples-cn} give the result for $d=0$ and $d=1$. If $X$ is of codimension $2$, then the third case in Example \ref{exples-cn} says that $H^{\pcn}_0(X)$ is free. In codimension $2$ again, we have $H^{\pcn}_1(X) = H^{\cn}_1(X)$, hence we are done by Lemma  \ref{lem:Hcn1}.
\end{proof}

\begin{remark}
If $\mathrm{codim}(X)=3$, then $H_1^{\pcn}(X)=H^{\cn}_1(X)\oplus H^{\cn}_3(X)$. Since $H^{\cn}_3(X)=\ker \delta_3$, Lemma  \ref{lem:Hcn1} also gives that $H_1^{\pcn}(X)$ is free. The combinatorics needed to prove that  $ H^{\cn}_2(X)$ -and  therefore $ H^{\pcn}_0(X)$- is free are much more involved. The torsion of conormal homology for manifolds of arbitrary codimension  will be studied   somewhere else.   
\end{remark}

\subsection{K\"unneth Formula for Conormal homology}

Taking advantage of the previous paragraph, we consider a product  $X=X_1\times X_2$ of two manifolds with corners, one of them being of codimension $\le 2$. It is understood that the defining functions used for $X$ are obtained by pulling back the ones used for $X_1$ and $X_2$. The tensor product $(\widehat{C}_*,\widehat{\delta})$ of the conormal complexes of $X_1$ and $X_2$ is given by 
\begin{equation}
 \widehat{C}_p=\bigoplus_{s+t=p} C_s(X_1)\otimes C_t(X_2)\quad  \text{ and } \quad\widehat{\delta}(x\otimes y) = \delta(x)\otimes y +(-1)^{t} x\otimes \delta(y)
\end{equation}
where $x\in C_t(X_1)$ in the second formula. We have an isomorphism of differential complexes:
\begin{equation}
  (\widehat{C}_*,\widehat{\delta})\simeq (C_*(X),\delta).
\end{equation}
It is given by the map
\begin{equation}
\Psi_p: \widehat{C}_p=\bigoplus_{s+t=p} C_s(X_1)\otimes C_t(X_2) \longrightarrow C_p(X)
\end{equation}
 defined by: 
\begin{equation}\label{Kunnethchainiso}
 (f\otimes \epsilon_{f})\otimes (g\otimes \epsilon_{g}) \longmapsto  (f\times g)\otimes \epsilon_{f}\cdot  \epsilon_{g},
\end{equation}
where we did not distinguish   differential forms on $X_j$ and their pull-back to $X$ via the canonical projections and $\cdot$ denotes again the exterior product. Since $H^{\cn}_*(X_j)$ is torsion free for $j=1$ or $2$ by assumption, we get by K\"unneth Theorem: 
\begin{equation}
 H_p(\widehat{C}_* , \hat{\delta}) = \bigoplus_{s+t=p}H_s^{\cn}(X_1)\otimes H_t^{\cn}(X_2).
\end{equation}
Therefore: 
\begin{proposition}[K\"unneth Formula]\label{Kunneth}
Assume that $X=X_1\times X_2$ with one factor at least of codimension $\le 2$. Then we have:
\begin{equation}
H_0^{\pcn}(X)\simeq H_0^{\pcn}(X_1)\otimes H_0^{\pcn}(X_2)\oplus H_1^{\pcn}(X_1)\otimes H_1^{\pcn}(X_2),
\end{equation}
\begin{equation}
H_1^{\pcn}(X)\simeq H_0^{\pcn}(X_1)\otimes H_1^{\pcn}(X_2)\oplus H_1^{\pcn}(X_1)\otimes H_0^{\pcn}(X_2)
\end{equation}
\end{proposition}
The following straightforward corollary will be useful later on:
\begin{corollary}
If $X=\Pi_iX_i$ is a finite product of manifold with corners $X_i$ with $codim(X_i)\leq 2 $, then the groups $H_*^{\pcn}(X)$ are torsion free.
\end{corollary}

The exact same arguments as above work to show that the Kunneth formula holds in full generality for conormal homology with rational coefficients, {\it i.e.} for $H_*^{\pcn}(X)\otimes_\mathbb{Z}\mathbb{Q}$. We state the proposition as we will use it later:

\begin{proposition}[K\"unneth Formula with rational coefficients]\label{KunnethQ}
For $X=X_1\times X_2$ we have:
\begin{equation}
H_0^{\pcn}(X)\otimes_\mathbb{Z}\mathbb{Q}\simeq (H_0^{\pcn}(X_1)\otimes_\mathbb{Z}\mathbb{Q})\otimes (H_0^{\pcn}(X_2)\otimes_\mathbb{Z}\mathbb{Q})\oplus (H_1^{\pcn}(X_1)\otimes_\mathbb{Z}\mathbb{Q})\otimes (H_1^{\pcn}(X_2)\otimes_\mathbb{Z}\mathbb{Q}),
\end{equation}
\begin{equation}
H_1^{\pcn}(X)\otimes_\mathbb{Z}\mathbb{Q}\simeq (H_0^{\pcn}(X_1)\otimes_\mathbb{Z}\mathbb{Q})\otimes (H_1^{\pcn}(X_2)\otimes_\mathbb{Z}\mathbb{Q})\oplus (H_1^{\pcn}(X_1)\otimes_\mathbb{Z}\mathbb{Q})\otimes (H_0^{\pcn}(X_2)\otimes_\mathbb{Z}\mathbb{Q})
\end{equation}
\end{proposition}

\section{The computation of $K_*(\cK_b(X))$}\label{secK_b}
 We keep all the notations and conventions of Section \ref{seccnhom}. In particular, the defining functions  induce  a trivialisation of the conormal bundle of any face $f$:
 \begin{equation}
  N(f)\simeq f\times E_f,
 \end{equation}
in which the $p$-dimensional real vector space $E_f$ inherits a basis $b_f=(e_i)_{i\in I}$, where $I$ is  characterized by $f\in c(I)$. These data induce an isomorphism
 \begin{equation}\label{Gamma-b-on-faces}
  \Gamma_b(X)|_f \simeq C^*(\cC(f) \times E_f)
 \end{equation}
where $\cC(f)$ denotes the pair groupoid over $f$, as well as a linear isomorphism $\varphi_f : \RR^p\to E_f$.
 
Also, the filtration \eqref{filtration-by-codimension} gives rise to the following filtration of the $C^*$-algebra $ \cK_b(X)= C^*(\Gamma_b(X))$ by ideals:
\begin{equation}\label{eq:K-filt}
   \cK(L^2(\overset{\circ}{X}))=A_0\subset A_1 \subset \ldots\ \subset A_d=A= \cK_b(X),
\end{equation} 
with $A_m= C^*(\Gamma(X)\vert_{X_m})$  for any $0\le m\le d$. The isomorphisms \eqref{Gamma-b-on-faces} induce
\begin{equation}\label{Alg-for-E1}
   A_{m}/A_{m-1}\simeq C^*(\Gamma\vert_{X_m\setminus X_{m-1}})\simeq  \bigoplus_{f\in F_m} C^*(\cC(f) \times E_f).
\end{equation}

\subsection{The first differential of the spectral sequence for $K_*(A)$}\label{subsecsseq}
The $K$-theory spectral sequence $(E^r_{*,*},d^r_{*,*})_{r\ge 1}$ 
  associated with  \eqref{eq:K-filt} \cite{schochet-top-method-I,kono-tamaki} converges to: 
\begin{equation}
   E^\infty_{p,q} = K_{p+q}(A_p)/ K_{p+q}(A_{p-1}).
\end{equation}
Here we have set  $K_n(A)=K_0(A\otimes C_0(\RR^n))$ for any $C^*$-algebra $A$. By construction, all the terms $E^r_{p,2q+1}$   vanish and by Bott periodicity, $E^r_{p,2q}\simeq E^r_{p,0}$. Also, all the differentials $d^{2r}_{p,q}$ vanish. By definition:
\begin{equation}
 d^1_{p,q} :  E^1_{p,q} = K_{p+q}(A_p/A_{p-1})\longrightarrow  E^1_{p-1,q}= K_{p+q-1}(A_{p-1}/A_{p-2})
\end{equation}
is the connecting homomorphism of the short exact sequence:
\begin{equation}\label{eq:exact-sequence-giving-d1}
  0 \longrightarrow A_{p-1}/A_{p-2}\longrightarrow A_{p}/A_{p-2}\longrightarrow A_{p}/A_{p-1}\longrightarrow 0.
\end{equation}
By \eqref{Alg-for-E1}, we get isomorphisms:
\begin{equation}\label{eq:E1-friendly-form}
  E^1_{p,q} \simeq  \bigoplus_{f\in F_p} K_{p+q}(C^*(\cC(f)\times E_f)) .
\end{equation}
Since the real vector space $E_f$ has dimension $p$, the groups $E^1_{p,q}$ vanish for odd $q$ and for even $q$, we have after applying Bott periodicity, \( E^1_{p,q}\simeq \ZZ^{\# F_p}\). 

Melrose and Nistor   \cite[Theorem 9]{MelNis} already achieved the computation of $d^1_{*,*}$.  In order  to relate the terms $E^2_{*,*}$ with the elementary defined conormal homology, we reproduce their computation in a slightly different way. Our approach is based on the next two lemmas. 
\begin{lemma}\label{lem:R1-acting-onR+}
 Let $\RR_+\rtimes \RR$ be the groupoid of the action of $\RR$ onto $\RR_+$ given by 
\begin{equation}\label{eq:basic-action}
   t.\lambda = te^\lambda,\qquad t\in \RR_+,\ \lambda\in \RR.
\end{equation}
The element $\alpha \in KK_1(C^*(\RR),C^*(\RR_+^*))$ associated with the exact sequence 
\begin{equation}\label{eq:exact-seq-basic-action}
  0\longrightarrow C^*(\cC(\RR_+^*))\longrightarrow C^*(\RR_+\rtimes \RR)\longrightarrow C^*(\RR)\longrightarrow 0
\end{equation}
is a $KK$-equivalence. 
\end{lemma}
\begin{proof}
 By the Thom-Connes isomorphism, the $C^*$-algebras $C^*(\RR_+\rtimes \RR)$ and $C^*(\RR_+\times \RR)$ are $KK$-equivalent. The latter being $K$-contractible, the result follows. 
\end{proof}
\begin{lemma}\label{lem:switching-coordinates-in-Rp-actions-1}
 Let $\RR_+\rtimes_{i}\RR^p$ be the groupoid  given by the action of  the $i^{\text{th}}$ coordinate of $\RR^p$  on $\RR_+$ by \eqref{eq:basic-action}. Let $\alpha_{i,p} \in KK_1(C^*(\RR^p), C^*(\RR^{p-1}))$ be the $KK$-element induced by   the exact sequence 
\begin{equation}\label{eq:exact-seq-basic-action-i-p}
  0\longrightarrow C^*(\cC(\RR_+^*)\times \RR^{p-1})\longrightarrow C^*(\RR_+\rtimes_{i} \RR^p)\longrightarrow C^*(\RR^p)\longrightarrow 0.
\end{equation}
Then for all $1\le i\le p$ we have
\begin{equation}
   \alpha_{i,p} = (-1)^{i-1}\alpha_{1,p} \text{ and }   \alpha_{1,p}= \sigma_{C^*(\RR^{p-1})}(\alpha),
\end{equation}
where $\sigma_D : K_*(A,B)\to K_*(A\otimes D,B\otimes D)$ denotes the Kasparov suspension map. 
\end{lemma}
\begin{proof}
  Let  $\tau$  be a permutation  of $\{1,2,\ldots,p\}$ and $i\in\{1,\ldots,p\}$. We denote in the same way   the corresponding   automorphisms of $\RR^p$ and   $C^*(\RR^p)$. We have a groupoid isomorphism 
\[
 \widetilde{\tau}: \RR_+\rtimes_{i}\RR^p\overset{\simeq}{\longrightarrow}\RR_+\rtimes_{\tau(i)}\RR^p
\]
and if we denote by $\tau_{i}$ the automorphism of $\RR^{p-1}$ obtained by removing the  $i^{\text{th}}$ factor in the domain of $\tau$ and the ${\tau(i)}^{\text{th}}$ factor in the range of $\tau$, we get a commutative diagram of exact sequences: 
\begin{equation}
\xymatrix{
0\ar[r]& C^*(\cC(\RR_+^*)\times \RR^{p-1}))\ar[d]^-{\tau_i}\ar[r]& C^*(\RR_+\rtimes_{i}\RR^p)\ar[d]^-{\widetilde{\tau}}\ar[r] & C^*(\RR^p)\ar[r]\ar[d]^-{\tau}&0\\
0\ar[r]& C^*(\cC(\RR_+^*)\times \RR^{p-1}))\ar[r]& C^*(\RR_+\rtimes_{\tau(i)}\RR^p)\ar[r]&C^*(\RR^p)\ar[r]& 0
}
\end{equation}
It follows that 
\begin{equation}
    \alpha_{\tau(i),p} = [\tau^{-1}]\otimes \alpha_{i,p}\otimes [\tau_i] \in KK_1(C^*(\RR^p),\cK\otimes C^*(\RR^{p-1})). 
\end{equation}
Taking $\tau=(1,i)$, we get $\tau=\tau^{-1}$ and $\tau_{i}=\id$, so that $\alpha_{i,p}=[\tau] \otimes \alpha_{1,p}$. Moreover, observe that for any $j$, 
\begin{equation}
 [(j-1,j)] = 1_{j-2}\otimes [f]\otimes 1_{p-j} \in \KK(C^*(\RR^p),C^*(\RR^p))
\end{equation}
where $[f]=-1\in KK(C^*(\RR^2),C^*(\RR^2))$ is the class of the flip automorphism and we have used the identification 
\[
 C^*(\RR^p)=C^*(\RR^{j-2})\otimes C^*(\RR^2)\otimes C^*(\RR^{p-j}).
\]
Using 
\[
  (1,i)=(1,2).(2,3)\ldots (i-1,i)
\]
now gives $[\tau]=(-1)^{i-1}$. Factorizing $C^*(\RR^{p-1})$ on the right in the sequence \eqref{eq:exact-seq-basic-action-i-p} for $i=1$ gives the  assertion $\alpha_{1,p}= \sigma_{C^*(\RR^{p-1})}(\alpha)$.
\end{proof}

\bigskip 
Using the canonical isomorphism $KK_1(C^*(\RR),C^*(\RR_+^*))\simeq KK_1(C_0(\RR),\CC)$, we can define a generator $\beta$  of $K_1(C_0(\RR))$ by
\begin{equation}\label{alphabeta}
  \beta\otimes \alpha = +1.
\end{equation}
For any $f\in F_p$ we then obtain a generator $\beta_f$ of $K_p(C_0(E_f))$ by 
\begin{equation}\label{Cliffrelation}
  \beta_f = (\varphi_f)_*(\beta^p)\in K_p(C_0(E_f))
\end{equation}
where  $\beta^p$ is the external product:
\begin{equation}
 \beta^p =\beta \otimes_{\CC}\cdots\otimes_{\CC}\beta \in K_p(C_0(\RR^p)).
\end{equation}
Picking up rank one projectors $p_f$ in $C^*(\cC(f))$, we get a basis of the free $\ZZ$-module $E^1_{p,0}$:
\begin{equation}
  (p_f\otimes \beta_f)_{f\in F_p}.
\end{equation}
Bases of $E^1_{p,q}$ for all even $q$ are deduced from the previous one by applying Bott periodicity. 

Now consider  faces $f\in F_p$ and $g\in F_{p-1}$ such that $f\subset \partial \overline{g}$. The $p$ and $p-1$ uples $I$, $J$ such that $f\in c(I)$ and $g\in c(J)$ differ by exactly one index, say the $j^{\text{th}}$, and we define
\begin{equation}
 \sigma(f,g)=(-1)^{j-1}.
\end{equation}
Introduce the exact sequence   
\begin{equation}\label{eq:exact-seq-basic-action-f-g}
  0\longrightarrow C^*(\cC(f\times\RR_+^*)\times E_g)\longrightarrow C^*(\cC(f)\times(\RR_+\rtimes_{j} E_f))\longrightarrow C^*(\cC(f)\times E_f)\longrightarrow 0,
\end{equation}
where $\RR_+\rtimes_{j} E_f$ denotes the transformation groupoid where the $j^{\text{th}}$ coordinate (only) of $E_f$ acts on $\RR_+$ by \eqref{eq:basic-action} again. We note 
\[
  \partial_{f,g} : K_{p}(C^*(\cC(f)\times E_f)) \longrightarrow K_{p-1}(C^*(\cC(g)\times E_g))
\]
the connecting homomorphism associated with \eqref{eq:exact-seq-basic-action-f-g},  followed by the unique $KK$-equivalence
\begin{equation}
  C^*(\cC(f\times\RR_+^*))\longrightarrow  C^*(\cC(g))
\end{equation}
provided by any tubular neighborhood of $f$ into $g$.
 
\begin{proposition}\label{heart-of-d1}
  With the notation above, we get 
  \begin{equation}
   \partial_{f,g}(p_f\otimes \beta_f) = \sigma(f,g). p_g\otimes \beta_g.
  \end{equation}
\end{proposition}
\begin{proof}
 Identify $E_f\simeq \RR^p$ and $E_g \simeq \RR^{p-1} $ using $b_f,b_g$ and apply Lemmas \ref{lem:switching-coordinates-in-Rp-actions-1} and \ref{lem:R1-acting-onR+}.
\end{proof}

\bigskip We can now achieve the determination of $d^1_{*,*}$.

\begin{theorem}\label{thmdiff1}
 We have 
\begin{equation}
  \forall f\in F_p,\qquad   d^1_{p,0}(p_f\otimes\beta_f) = \sum_{\substack{g\in F_{p-1} \\ f\subset \partial \overline{g}}}
     \sigma(f,g)p_{g}\otimes \beta_g.
\end{equation}
\end{theorem}
\begin{proof}
For $p=0$, we have $F_{p-1}=\emptyset$ and $d^1_{p,0}=0$, the result follows. For $p\ge 1$, we recall that 
\begin{equation}
  d^1_{p,0} : \oplus_{f\in F_p}  K_{p}(C^*(\cC(f)\times E_f)) \longrightarrow  \oplus_{g\in F_{p-1}} K_{p-1}(C^*(\cC(g)\times E_g) ).
\end{equation}
is the connecting homomorphism in $K$-theory of the exact sequence \eqref{eq:exact-sequence-giving-d1}. We obviously have  
\begin{equation}
 d^1_{p,0}(p_f\otimes\beta_f) = \sum_{g\in F_{p-1}} \partial_g(p_f\otimes\beta_f)
\end{equation}
where $\partial_g$ is the connecting homomorphism in $K$-theory of the exact sequence 
\begin{equation}
  0\longrightarrow C^*(\Gamma\vert_{g})\longrightarrow C^*(\Gamma\vert_{g\cup f})\longrightarrow C^*(\Gamma\vert_{f})\longrightarrow 0.
\end{equation}
If $f\not\subset \partial \overline{g}$ then the sequence  splits and $\partial_g(p_f\otimes \beta_f)=0$. Let $g\in F_{p-1}$ be such that $f \subset \partial \overline{g}$.
Let $\cU$ be an open neighborhood of $f$ in $X$   such that there exists a diffeomorphism  
 \begin{equation}\label{collar-ident}
  \cU_g:= \cU\cap g \longrightarrow f\times (0,+\infty),\qquad x \longmapsto (\phi(x),r_{i_g}(x)),
 \end{equation}
 where $r_{i_g}$ is the defining function of $f$ in $\overline{g}$. This yields a commutative diagram
\begin{equation}\label{diag:localization-boundary-around-face}
\xymatrix{
0\ar[r]& C^*(\Gamma\vert_{\cU_g})\ar[d]^-{\hookrightarrow}_{\iota}\ar[r]& C^*(\Gamma\vert_{\cU_g\cup f})\ar[d]^-{\hookrightarrow}\ar[r] & 
C^*( \Gamma\vert_{f})\ar[r]\ar[d]^-{=}& 0 \\
0\ar[r]& C^*(\Gamma\vert_{g})\ar[r]& C^*(\Gamma\vert_{g\cup f})\ar[r]& C^*( \Gamma\vert_{f})\ar[r]& 0
}
\end{equation}
whose upper   sequence coincides with \eqref{eq:exact-seq-basic-action-f-g}  using \eqref{collar-ident}. This implies
\begin{equation}
  \partial_g = \partial_{f,g}.
\end{equation}
The result follows by Proposition \ref{heart-of-d1}.
\end{proof}

\bigskip The map $d^1_{p,q}$, $q$ even, is deduced from $d^1_{p,0}$ by Bott periodicity. We are ready to relate the $E^2$ terms with conormal homology.

\begin{corollary}\label{cordiff1}
 For every $p\in \{1,...,d\}$  there are isomorphisms 
\begin{equation}\label{proofPCHvsKthr=1:HR1}
\phi_{p,1}^i : H_{\mathrm{i}}^{\pcn}(X_p,X_{p-1}) \longrightarrow   K_i(A_p/A_{p-1}), \qquad \{0,1\}\ni i\equiv p\mod 2,
\end{equation}
such that the following diagram commute 
\begin{equation}\label{proofPCHvsKthr=1:HR2}
 \xymatrix{
 H_{\mathrm{i}}^{\pcn}(X_p,X_{p-1})\ar[r]^-{\phi_{p,1}^i}  \ar[d]_{\partial} & K_p(A_p/A_{p-1})) \ar[d]^{d^1_{p,0}} \\
 H_{\mathrm{1-i}}^{\pcn}(X_{p-1},X_{p-2})\ar[r]^{\phi_{p-1,1}^{1-i}}  & K_{p-1}(A_{p-1}/A_{p-2})) 
 } 
\end{equation}
where $\partial$ stands for the connecting morphism  in conormal homology.
\end{corollary}

\begin{proof}
If $i\equiv p\mod 2$ then $H_{\mathrm{i}}^{\pcn}(X_p,X_{p-1})=C_p(X)$ and $\partial=\delta_p:C_p(X)\longrightarrow C_{p-1}(X)$. We define \eqref{proofPCHvsKthr=1:HR1} by $\phi_{p,1}^i(f\otimes \epsilon_f)=p_f\otimes \beta_f$  and Theorem \ref{thmdiff1} gives the commutativity of \eqref{proofPCHvsKthr=1:HR2}.
\end{proof}

\bigskip In other words, the map $f\otimes \epsilon_f\mapsto p_f\otimes \beta_f$ induces a isomorphism
\begin{equation}
   H_{p}^{\cn}(X) \simeq E^2_{p,0}.
\end{equation}
It would be very interesting to compute the higher differentials $d^{2r+1}_{p,0}$.

\subsection{The final computation for $K_*(\cK_b(X))$ in terms of conormal homology}

Before getting to the explicit computations and to the analytic corollaries in term of these, let us give a simple but interesting result. It is about the full understanding of the six term exact sequence in K-theory of the fundamental sequence 
\begin{equation}\label{sesbcompact}
\xymatrix{
0\ar[r]&\cK(X)\ar[r]^-{i_0}&\cK_b(X)\ar[r]^-{r}&\cK_b(\partial X) \ar[r]&0.
}
\end{equation}

\begin{proposition}\label{thmbcompact}
For a connected manifold with corners X of codimension greater or equal to one the induced morphism by $r$ in $K_0$, $r:K_0(\cK_b(X))\to K_0(\cK_b(\partial X))$, is an isomorphism. Equivalently
\begin{enumerate}
\item The morphism $i_F:K_0(\cK)\cong \mathbb{Z} \to K_0(\cK_b(X))$ is the zero morphism.
\item The connecting morphism $K_1(\cK_b(\partial X)) \to K_0(\cK)\cong \mathbb{Z}$ is surjective.
\end{enumerate}
\end{proposition}

\begin{proof}
Let $X$ be a connected manifold with corners of codimension $d$. With the notations of the last section, the sequence \ref{sesbcompact} correspond to the canonical sequence
$$
\xymatrix{
0\ar[r]&A_0\ar[r]&A_d\ar[r]&A_d/A_0\ar[r]&0.
}
$$ 
We will prove that the connecting morphism $K_1(A_d/A_0) \to K_0(A_0)\cong \mathbb{Z}$ is surjective.
The proof will proceed by induction, the case $d=1$ immediately satisfies this property. So let us assume that the connecting morphism $K_1(A_{d-1}/A_0) \to K_0(A_0)$ associated to the short exact sequence 
$$\xymatrix{
0\ar[r]&A_0\ar[r]&A_{d-1}\ar[r]&A_{d-1}/A_0\ar[r]&0.
}$$
is surjective. Consider now the following commutative diagram of short exact sequences
\begin{equation}\label{AdA0}
\xymatrix{
0\ar[r]&0\ar[r]&A_d/A_{d-1}\ar[r]&A_d/A_{d-1}\ar[r]&0\\
0\ar[r]&A_0 \ar[u]\ar[r]&A_d\ar[u]\ar[r]&A_d/A_0\ar[u]\ar[r]&0\\
0\ar[r]&A_0 \ar[u]\ar[r]&A_{d-1}\ar[u]\ar[r]&A_{d-1}/A_0\ar[u]\ar[r]&0.
}
\end{equation}
By applying the six-term short eaxt sequence in K-theory to it we obtain that the following diagram, where $\partial_d$ and 
$\partial_{d-1}$ are the connecting morphisms associated to the middle and to the bottom rows respectively,
\[
\xymatrix{
K_(A_d/A_0)\ar[rd]^-{\partial_d}&\\
K_1(A_{d-1}/A_0)\ar[u]\ar[r]_-{\partial_{d-1}}& K_0(A_0)
}
\]
is commutative. Hence, by inductive hypothesis, we obtain that $\partial_d$ is surjective.
\end{proof}

\begin{remark} 
Roughly speaking, the previous proposition tells us that the analytical index of a fully elliptic element carries no information about its Fredholm index, this information being completely contained in   some element  of 
$K_1(\cK_b(\partial X))$.
\end{remark}

We have next our main K-theoretical computation:

\begin{theorem}\label{thmPCHvsKth}
Let $X$ be a finite product of manifolds with corners of codimension less or equal to three. There are natural isomorphisms 
\begin{equation}
\xymatrix{
H_{\mathrm0}^{\pcn}(X)\otimes_\mathbb{Z}\mathbb{Q}\ar[r]^-{\phi_X}_-\cong & K_0(\cK_b(X)))\otimes_\mathbb{Z}\mathbb{Q}
}
\text{ and }\qquad 
\xymatrix{
H_{\mathrm1}^{\pcn}(X)\otimes_\mathbb{Z}\mathbb{Q}\ar[r]^-{\phi_X}_-\cong & K_1(\cK_b(X))\otimes_\mathbb{Z}\mathbb{Q}.
}
\end{equation}
In the case $X$ contains a factor of codimension at most two or $X$ is of codimension three the result holds even without tensoring by $\mathbb{Q}$.
\end{theorem}

\begin{proof}

{\bf 1A. $Codim(X)=0$:} The only face of codimension 0 is $\intx$ (we are always assuming $X$ to be connected). The isomorphism 
$$H_0^{\cn}(X_0) \stackrel{\phi_0}{\longrightarrow} K_0(A_0)$$
is simply given by sending $\intx$ to the rank one projector $p_{\intx}$ chosen in section \ref{subsecsseq}.  

{\bf 1B. $Codim(X)=1$:} Consider the canonical short exact sequence
$$
\xymatrix{
0\ar[r]&A_0\ar[r]&A_1\ar[r]&A_1/A_0\ar[r]&0
}
$$
That gives, since $d^1_{1,0}$ is surjective, the following exact sequence in K-theory
$$
\xymatrix{
0\ar[r]&K_1(A_1)\ar[r]&K_1(A_1/A_0)\ar[r]^-{d^1_{1,0}}&K_0(A_0)\ar[r]&0
}
$$
from which $K_1(A_1)\cong \ker d^1_{1,0}$ and $K_0(A_1)=0$ (since $K_0(A_1/A_0)=0$ by a direct computation for K-theory or for conormal homology). By theorem \ref{thmdiff1}, corollary \ref{cordiff1}, we have the following commutative diagram
$$
\xymatrix{
K_1(A_1/A_0)\ar[r]^-{d^1_{1,0}}&K_0(A_0)\\
H_1^{\pcn}(X_1\setminus X_0)\ar[u]^-{\phi_{1,0}}_-{\cong}\ar[r]_-{\delta_1}&H_0^{\pcn}(X_0)\ar[u]^-{\phi_0}_-{\cong}.
}
$$

Then there is a unique natural isomorphism 
$$H_1^{\pcn}(X_1) \stackrel{\phi_1}{\longrightarrow} K_1(A_1),$$
fitting the following commutative diagram
$$
\xymatrix{
0\ar[r]&K_1(A_1)\ar[r]&K_1(A_1/A_0)\ar[r]^-{d^1_{1,0}}&K_0(A_0)\ar[r]&0\\
0\ar[r]&H_1^{\pcn}(X_1)\ar[u]^-{\phi_1}_-{\cong}\ar[r]&H_1^{\pcn}(X_1\setminus X_0)\ar[u]^-{\phi_{1,0}}_-{\cong}\ar[r]^-{\partial_{1,0}}&H_0^{\pcn}(X_0)\ar[u]^-{\phi_0}_-{\cong}\ar[r]&0.
}
$$
{\bf 1C. $Codim(X)=2$:} 
We first proof that we have natural isomorphisms
\begin{equation}\label{eq1-propinductionstep}
\xymatrix{ H_{*}^{\cn}(X_l,X_m) \ar[r]^-{\phi_{l,m}}_-{\cong}& K_*(A_l/A_m)
}
\end{equation}
for every $0\leq m\leq l$ with $l-m=2$ and for every manifold with corners (of any codimension). Indeed, this case can be treated very similar to the above one. Suppose $l$ is even, the odd case is treated in the same way by exchanging $K_0$ by $K_1$ and $H_0$ by $H_1$. By comparing the long exact sequences in conormal homology we have that there exist unique natural isomorphisms $\phi_{l,l-2}^0$ and $\phi_{l,l-2}^1$ making the following diagram commutative
$$
{\tiny
\xymatrix{
0\ar[r]&K_0(A_l/A_{l-2})\ar[r]&K_0(A_l/A_{l-1})\ar[r]^-{d^1_{l,0}}&K_1(A_{l-1}/A_{l-2})\ar[r]&K_1(A_l/A_{l-2})\ar[r]&0\\
0\ar[r]&H_0^{\pcn}(X_l\setminus X_{l-2})\ar[u]^-{\phi_{l,l-2}^0}_-{\cong}\ar[r]&H_0^{\pcn}(X_l\setminus X_{l-1})\ar[u]^-{\phi_{l,l-1}}_-{\cong}\ar[r]^-{\partial_{l,0}}&H_1^{\pcn}(X_{l-1}\setminus X_{l-2})\ar[u]^-{\phi_{l-1,l-2}}_-{\cong}\ar[r]&H_1^{\pcn}(X_l\setminus X_{l-2})\ar[u]^-{\phi_{l,l-2}^1}_-{\cong}\ar[r]&0.
}}
$$
since the diagram in the middle is commutative again by corollary \ref{cordiff1}.

Let us now pass to the case when $codim(X)=2$. Consider the short exact sequence:
\begin{equation}
\xymatrix{
0\ar[r]&A_0\ar[r]&A_2\ar[r]&A_2/A_0\ar[r]&0.
}
\end{equation}

We compare its associated six term short exact sequence in $K$-theory with the one in conormal homology to get
\begin{equation}\label{PCNvsKthdiagcodim2}
{\tiny\xymatrix{
&\mathbb{Z}\ar[rr]&&K^0(A_2)\ar[rr]&&K_0(A_2/A_0)\ar[dd]&\\
H_0(X_0)\ar[rr]\ar[ru]^-{\phi_0}_-{\cong}&&H_0^{\pcn}(X_2)\ar[rr]\ar@{.>}[ru]|-{?_2}&&H_0^{\pcn}(X_2,X_0)\ar[dd]\ar[ru]\ar@{.>}[ru]^-{\phi_{2,0}}_-{\cong}&&\\
&K_1(A_2/A_0)\ar@{.>}[uu]&&K_1(A_2)\ar@{.>}[ll]&&0\ar@{.>}[ll]&\\
H_1^{\pcn}(X_2,X_0)\ar[uu]\ar@{.>}[ru]^-{\phi_{2,0}}_-{\cong}&&H_1^{\pcn}(X_2)\ar@{.>}[ru]|-{?_1}\ar[ll]&&0\ar[ru]\ar[ll]&&
}}
\end{equation}
where we need now to define isomorphisms $?_1$ and $?_2$. In fact if we can define morphims such that the diagrams are commutative then by a simple Five lemma argument they would be isomorphisms. The first thing to check is that 
\begin{equation}\label{phi20diag}
\xymatrix{
K_1(A_2/A_0)\ar[r]^-{d_{2,0}}&K_0(A_0)\cong \mathbb{Z}\\
H_1^{\pcn}(X_2,X_0)\ar[u]^-{\phi_{2,0}}_-{\cong}\ar[r]_-{\partial_{2,0}}& H_0^{\pcn}(X_0)\ar[u]_-{\phi_{0}}^-{\cong}
}
\end{equation}
is commutative. Indeed, this can be seen by considering the following commutative diagram of short exact sequences
\begin{equation}
\xymatrix{
0\ar[r]&0\ar[r]&A_2/A_1\ar[r]&A_2/A_1\ar[r]&0\\
0\ar[r]&A_0\ar[u]\ar[r]&A_2\ar[u]\ar[r]&A_2/A_0\ar[u]\ar[r]&0\\
0\ar[r]&A_0\ar[u]\ar[r]&A_1\ar[u]\ar[r]&A_1/A_0\ar[u]\ar[r]&0,
}
\end{equation}
applying the associated diagram between the short exact sequences that gives that the connecting morphism for the middle row, 
$K_1(A_2/A_0)\stackrel{d_{2,0}}{\longrightarrow}K_0(A_0)$, is given by a (any) splitting of $K_1(A_1/A_0)\to K_1(A_2/A_0)$ (both modules are free $\mathbb{Z}$-modules by theorem \ref{freePCN}) followed by the connecting morphism associated to the exact sequence on the bottom of the above diagram. By definition of $\phi_{2,0}$ in (\ref{eq1-propinductionstep}) above and by corollary \ref{cordiff1} we have that these two last morphisms are compatible with the analogs in the respective conormal homologies, since the connecting morphism $\partial_{2,0}$ in conormal homology is obtained in this way as well we conclude the commutativity of (\ref{phi20diag}).
We are ready to define $?_1$ and $?_2$. For the first one, $?_1$, there is a unique isomorphism $\phi_2^1$ fitting the following commutative diagram
$$
\xymatrix{
0\ar[r]&K_1(A_2)\ar[r]&K_1(A_2/A_0)\\
0\ar[r]&H_1^{\pcn}(X_2)\ar[u]^-{\phi_2^1}_-{\cong}\ar[r]&H_1^{\pcn}(X_2,X_0)\ar[u]^-{\phi_{2,0}^1}_-{\cong}
}$$
and given by restriction of $\phi_{2,0}^1$ to the image of $H_1^{\pcn}(X_2)\to H_1^{\pcn}(X_2,X_0)$. 
Now, for defining $?_2$ we have by theorem \ref{thmbcompact} an unique isomorphism $\phi_{2}^0$ fitting the following diagram
$$
\xymatrix{
K_0(A_2)\ar[r]^-{\cong}&K_0(A_2/A_0)\\
H_0^{\pcn}(X_2)\ar[u]^-{\phi_{2}^0}_-{\cong}\ar[r]_-{\cong}&H_0^{\pcn}(X_2,X_0)\ar[u]^-{\phi_{2,0}^0}_-{\cong}.
}$$
{\bf 1D. $Codim(X)=3$:} Consider the short exact sequence:
$$
\xymatrix{
0\ar[r]&A_2\ar[r]&A_3\ar[r]&A_3/A_2\ar[r]&0.
}
$$
We compare its associated six term short exact sequence in $K$-theory with the one in conormal homology to get
\begin{equation}\label{PCNvsKthdiagcodim3}
{\tiny\xymatrix{
&K_0(A_2)\ar[rr]&&K_0(A_3)\ar[rr]&&0\ar[dd]&\\
H_0^{\pcn}(X_2)\ar[rr]\ar[ru]^-{\phi_2}_-{\cong}&&H_0^{\pcn}(X_3)\ar[rr]\ar@{.>}[ru]|-{?_2}&&0\ar[dd]\ar[ru]\ar@{.>}[ru]&&\\
&K_1(A_3/A_2)\ar@{.>}[uu]&&K_1(A_3)\ar@{.>}[ll]&&K_1(A_2)\ar@{.>}[ll]&\\
H_1^{\pcn}(X_3,X_2)\ar[uu]\ar@{.>}[ru]^-{\phi_{3,2}}_-{\cong}&&H_1^{\pcn}(X_3)\ar@{.>}[ru]|-{?_1}\ar[ll]&&H_1^{\pcn}(X_2)\ar[ru]^-{\phi_2}_-{\cong}\ar[ll]&&
}}
\end{equation}
where we need now to define isomorphisms $?_1$ and $?_2$. Again, if we can define morphims such that the diagrams are commutative then by a simple Five lemma argument they would be isomorphisms. Let us first check that the diagram
\begin{equation}\label{phi32diag}
\xymatrix{
K_1(A_3/A_2)\ar[r]^-{\partial}&K_0(A_2)\\
H_1^{\pcn}(X_3,X_2)\ar[u]^-{\phi_{3,2}}_-{\cong}\ar[r]_-{\partial}& H_0^{\pcn}(X_2)\ar[u]_-{\phi_2}^-{\cong}
}
\end{equation}
is commutative. For this consider the following commutative diagram of short exact sequences
\begin{equation}
\xymatrix{
0\ar[r]&A_1\ar[r]\ar[d]&A_1\ar[r]\ar[d]&0\ar[d]\ar[r]&0\\
0\ar[r]&A_2\ar[d]\ar[r]&A_3\ar[d]\ar[r]&A_3/A_2\ar[d]\ar[r]&0\\
0\ar[r]&A_2/A_1\ar[r]&A_3/A_1\ar[r]&A_3/A_2\ar[r]&0.
}
\end{equation}
that implies that the connecting morphism $K_1(A_3/A_2)\stackrel{\partial}{\rightarrow}K_0(A_2)$ followed by the morphism $K_0(A_2)\to K_0(A_2/A_1)$ coincides with the connecting morphism $K_1(A_3/A_2)\stackrel{\partial}{\rightarrow}K_0(A_2/A_1)$. Now, the two latter morphisms are compatible with the analogs in conormal homology via the isomorphisms described above and the morphism $K_0(A_2)\to K_0(A_2/A_1)$ is injective (since $K_0(A_1)=0$), hence the commutativity of diagram \ref{phi32diag} above follows. From diagram (\ref{PCNvsKthdiagcodim3}), by passage to the quotient, there is unique isomorphism $\phi_{3}^{0}$ (the one filling $?_2$ in the above diagram) such that 
$$
\xymatrix{
K_0(A_2)\ar[r]&K_0(A_3)\ar[r]&0\\
H_{0}^{\pcn}(X_2)\ar[u]^-{\phi_{2}^{0}}_-{\cong}\ar[r]&H_0^{\pcn}(X_3)\ar[u]^-{\phi_{3}^{0}}_-{\cong}\ar[r]&0
}$$
is commutative. Finally, for defining $?_1$, it is now enough to choose splittings  for the map
$$0\to H_1^{\pcn}(X_2) \to H_1^{\pcn}(X_3),$$
which is possible since $H_1^{\pcn}(X_3)$ is free (see theorem \ref{freePCN} and the remark below it)
and for the map
$$K_1(A_3)\to \im j\to 0,$$
where $j$ is the canonical morphism $j:K_1(A_3)\to K_1(A_3/A_2)$ (remember all the groups $K_*(A_p/A_{p-1})$ are torsion free).
 
{\bf 1E. If $X=\Pi_iX_i$ is a finite product with $codim(X_i)\leq 3$ and with at least one factor of codimension at most 2:} In this case the result would follow, by all the points above, if both, Periodic conormal homology and K-theory, satisfy the K\"unneth formula. Since the algebras $\cK_b(X)$ are nuclear because the groupoids $\Gamma_b(X)$ are amenable we have the K\"unneth formula in K-theory for these kind of algebras. Now, for conormal homology we verified the K\"unneth formula in proposition \ref{Kunneth}.

{\bf 1F. If $X=\Pi_iX_i$ is a finite product with $codim(X_i)\leq 3$, $\forall i$:} In this case the result holds rationnaly by the same arguments as above by using propostion \ref{KunnethQ}.
\end{proof}

\section{Fredholm perturbation properties and Euler conormal characters}

The previous  results yield a criterium for Property $(\cH\cF\cP)$ in terms of the Euler characteristic for conormal homology.
To fit with the assumptions of Theorem \ref{thmPCHvsKth}, we consider a manifold with corners $X$ of codimension $d$, which is given by the cartesian product of manifolds with corners of codimension at most $3$. 

\begin{definition}[Corner characters]\label{cornercharactersintro}
Let $X$ be a manifold with corners. We define the even conormal character of $X$ as 
\begin{equation}
\chi_0(X)=dim_\mathbb{Q}\,H_{0}^{\pcn}(X)\otimes_\mathbb{Z}\mathbb{Q}.
\end{equation}
Similarly, we define the odd conormal character of $X$ as
\begin{equation}
\chi_1(X)=dim_\mathbb{Q}\,H_{1}^{\pcn}(X)\otimes_\mathbb{Z}\mathbb{Q}.
\end{equation}
\end{definition}

We can consider as well
\begin{equation}
\chi(X)=\chi_0(X)-\chi_1(X),
\end{equation}	
then we have (by the rank nullity theorem)
\begin{equation}
\chi(X)=1-\# F_1+\# F_2-\cdots +(-1)^d\#F_d.
\end{equation}
We refer to the integer $\chi(X)$ as the Euler corner character of $X$. These numbers are clearly invariant under the natural notion of isomorphism of manifolds with corners. Their computation is elementary in any concrete situation.

In particular one can rewrite the theorem above to have, for $X$ as in the statement,

\begin{equation}\label{tableKthprodintro}
{\large 
\xymatrix{
&K_0(\cK_b(X))\otimes_\mathbb{Z}\mathbb{Q}\cong\mathbb{Q}^{\chi_0(X)}&\\
&K_1(\cK_b(X))\otimes_\mathbb{Z}\mathbb{Q}\cong\mathbb{Q}^{\chi_1(X)}&
}
}
\end{equation}
and, in terms of the corner character,
\begin{equation}
{\large \chi(X)=rank(K_0(\cK_b(X))\otimes_\mathbb{Z}\mathbb{Q})-rank(K_1(\cK_b(X))\otimes_\mathbb{Z}\mathbb{Q}).}
\end{equation}

In the case $X$ is a finite product of manifolds with corners of codimension at most 2  we even have
 \begin{equation}\label{tableKth}
K_0(\cK_b(X))\simeq \mathbb{Z}^{\chi_0(X)} \qquad \text{ and } \qquad 
  K_1(\cK_b(X))\simeq \mathbb{Z}^{\chi_1(X)}
\end{equation}
and also 
\(
 \chi_\cn(X)=\mathrm{rk}(K_0(\cK_b(X)))-\mathrm{rk}(K_1(\cK_b(X)))
\).

\vspace{2mm}
 
We  end with the  characterization of Property $(\cH\cF\cP)$ in terms of  conormal characteristics.    
\begin{theorem}\label{thmFPcornercycles}
Let $X$ be a compact connected manifold with corners of codimension greater or equal to one. If $X$ is a finite product of manifolds with corners of codimension less or equal to three we have that 
\begin{enumerate}
\item If $X$ satisfies the Fredholm Perturbation property then the even Euler corner character of $X$ vanishes, {\it i.e. $\chi_0(X)=0$}.
\item If the even Periodic conormal homology group vanishes, {\it i.e. $H_0^{\pcn}(X)=0$} then $X$ satisfies the stably homotopic Fredholm Perturbation property.
\item If $H_0^{\pcn}(X)$ is torsion free and if the even Euler corner character of $X$ vanishes, {\it i.e. $\chi_0(X)=0$} then $X$ satisfies the stably homotopic Fredholm Perturbation property.
\end{enumerate}
\end{theorem}

\bigskip \noindent \begin{proof}
\begin{enumerate}
\item Suppose $\chi_0(X)\neq 0$ then $K_0(\cK_b(X))\otimes_\mathbb{Z}\mathbb{Q}\cong\mathbb{Q}^{\chi_0(X)}$ is not the zero group. By theorem \ref{AnavsFredthm1} it is enough to prove that the rationalized analytic indicial index morphism
$$Ind_a:K^0_{top}(^bT^*X)\otimes_\mathbb{Z}\mathbb{Q}\longrightarrow K_0(\cK_b(X))\otimes_\mathbb{Z}\mathbb{Q}$$
is not the zero morphism. In \cite{MontNis} (theorem 12, 13 and proposition 7 in ref.cit.), Monthubert and Nistor construct a manifold with corners $Y$and a closed embedding of manifolds with embedded corners $X\stackrel{i}{\longrightarrow} Y$ to obtain a commutative diagram
\begin{equation}
\xymatrix{
K^0_{top}(^bT^*X)\otimes_\mathbb{Z}\mathbb{Q}\ar[d]_-{i!}\ar[r]^-{Ind_a}&K_0(\cK_b(X))\otimes_\mathbb{Z}\mathbb{Q}\ar[d]_-{\cong}^-{i_*}\\
K^0_{top}(^bT^*Y)\otimes_\mathbb{Z}\mathbb{Q}\ar[r]_-{Ind_a}^-{\cong}&K_0(\cK_b(Y))\otimes_\mathbb{Z}\mathbb{Q}.
}
\end{equation}
They call such a $Y$ a classifying space of $X$. For our purposes it would be then enough to show that the morphism $$i!:K^0_{top}(^bT^*X)\otimes_\mathbb{Z}\mathbb{Q}\longrightarrow K^0_{top}(^bT^*Y)\otimes_\mathbb{Z}\mathbb{Q}$$
is not the zero morphism. But now we are at the topological K-theory level (with compact supports) where classic topological arguments apply to get that the morphism above is not the zero morphism. Indeed, for construct $i!$ one uses a tubular neighborhood (which exist in this setting, see for example Douady \cite{Dou}), the first step is then a Thom isomorphism followed by a morphism induced by a classic extension by zero. This is summarized in proposition 5 in \cite{MontNis}. The conclusion follows.
\item If $H_{\mathrm0}^{\mathrm{pcn}}(X)=0$ then $H_{\mathrm0}^{\mathrm{pcn}}(X)\otimes_\mathbb{Z}\mathbb{Q}=0$ and the result follows from theorems \ref{thmPCHvsKth} and \ref{AnavsFredthm1}.
\item In this case $K_0(\cK_b(X))\cong \mathbb{Z}^{\chi_0(X)}$ by theorem \ref{thmPCHvsKth} and the arguments applied in the last two points identically apply to get the result (the results of Monthubert-Nistor cited above hold over $\mathbb{Z}$).
\end{enumerate}
\end{proof}

 \section{Appendix: more on conormal homology}\label{appendix}
 
We reproduce the discussion leading to the definition of the conormal differential in a slightly more general way. We keep the same notations.  Let $f\in F_p$,  $\epsilon_f$ an orientation of  $N(f)$ and  $g\in F_{p - k}$  such that $f\subset \overline{g}$.  The face $f$ is  characterized in $\overline{g}$ by the vanishing of $k$ defining functions  and we denote by $(g,f)$ the corresponding $k$-uple of their indices. Then the contraction  $\epsilon_g := e_{(g,f)}\lrcorner \epsilon_f $  is an orientation of $N(g)$. Recall that:
\begin{equation}
  e_{J}\lrcorner \cdot =  e_{j_1}\lrcorner( \ldots \lrcorner (e_{j_k} \lrcorner \cdot)\ldots).
\end{equation}    
For any integers $0\le k\le p$,  we define  $\delta^{k}_p: C_p(X)\to C_{p-k}(X)$ by 
\begin{equation}\label{diffpcn1}
   \delta^{k}_p(f\otimes \varepsilon_f) = \sum_{\substack{g\in F_{p-k},  \\ f\subset\overline{g}}} g\otimes e_{(g,f)}\lrcorner\varepsilon_f.
\end{equation}
We get a homomorphism $\delta^{\pcn} : \cC(X)\to \cC(X)$ of degree $1$ with respect to the $\ZZ_2$-grading by setting: 
\begin{equation}
 \delta^{\pcn}_i = \sum_{\substack{k\ge 0, \\ p\equiv i\, \mathrm{mod}\, 2}} \delta^{2k+1}_p,\qquad i=0,1. 
\end{equation}
 \begin{proposition}\label{propdeltadelta=0}
The map $\delta^{\pcn}$ is a differential, that is $\delta^{\pcn}\circ \delta^{\pcn} =0$.
\end{proposition}
\begin{proof}
Let $f\in F_p(X)$ and $\epsilon$ be an orientation of $N(f)$. We have
 \begin{equation}\label{pf:propdeltadelta=0-1}
 \delta^{\pcn}(\delta^{\pcn}(f\otimes \epsilon)) =
 \sum_{\substack{g,h\text{ s.t. } \overline{h} \supset\overline{g} \supset  f   \\ (g,f),(h,g)  \text{ are odd} }}\left(  h\otimes e_{(h,g)}\lrcorner(e_{(g,f)}\lrcorner\varepsilon)\right).
\end{equation}
Let $g,h$ providing a term in the sum above and denote  by $I,J,K$ the uples labelling the defining functions of $f,g,h$ respectively. Then set 
\begin{equation}
 J'=I\setminus (h,g).
\end{equation}
By  definition of manifolds with (embedded) corners, $H_{J'}$ is not empty and there exists a unique face $g'\in c(J')$ with $f\subset\overline{g'}$. This face $g'=\iota(g,h,f)$ satisfies the following properties:
\begin{itemize}
 \item  $f\subset \overline{g'}\subset \overline{h}$,
 \item  $(g',f)=(h,g)$ and $(h,g')=(g,f)$ are odd,
 \item  $\iota(g',h,f)=g$.
\end{itemize}
Finally, note that $\#(g,f)\not=\#(h,g)$, otherwise we would have $(h,f)=(h,g)+(g,f)$ even. This implies in particular that $g\not=g'$. These observations allow to reorganize the sum \eqref{pf:propdeltadelta=0-1} as follow:
 \begin{eqnarray*}
 \delta^{\pcn}(\delta^{\pcn}(f\otimes \epsilon)) &=& 
 \sum_{\substack{g,h\text{ s.t. } \overline{h} \supset\overline{g} \supset  f   \\ \#(g,f)<\#(h,g)  \text{ odd} }}\left(  h\otimes (e_{(h,g)}\lrcorner(e_{(g,f)}\lrcorner\varepsilon+ e_{(h,g')}\lrcorner(e_{(g',f)}\lrcorner\varepsilon))\right).
\end{eqnarray*}
Now 
\begin{eqnarray*}
 e_{(h,g)}\lrcorner(e_{(g,f)}\lrcorner\varepsilon)+ e_{(h,g')}\lrcorner(e_{(g',f)}\lrcorner\varepsilon) 
 &=& e_{(h,g)}\lrcorner(e_{(g,f)}\lrcorner\varepsilon)+ e_{(g,f)}\lrcorner(e_{(h,g)}\lrcorner\varepsilon)
 = 0 
\end{eqnarray*}
since $\#(g,f)$ and $\#(h,g)$ are odd. 
\end{proof}

\bigskip \noindent Proposition \ref{propdeltadelta=0} implies $\delta^1_{p-1}\circ\delta^1_p=0$ for any $p$. Since $\delta^1_*=\delta_*$, this proves the claim of Paragraph \ref{subsec:def-cn}. Moreover:
\begin{proposition}\label{prop:contracting-delta}
 The identity map $(\cC_*(X),\delta^1)\longrightarrow (\cC_*(X),\delta) $  induces an isomorphism between the $\ZZ_2$-graded homology groups. 
\end{proposition}
 \begin{lemma}
 The following equality hold for any $k\ge 0$:
 \begin{equation}
   \delta^{2k+1} = \delta^{2k}\circ \delta^1 = \delta^1\circ \delta^{2k}
   \end{equation}
\end{lemma}
\begin{proof}[of the lemma]
 Let $f$ be a codimension $p$ face and $\epsilon$ an orientation of $N(f)$.
 Let $I$ be the $p$-uple defining $f$. Then $g$ is a face such that $f\subset \overline{g}$ if and only if  $g$  is a connected component of $H_J$ for some $J\subset I$. Since the definition of $\delta(f)$ only involves faces $g$ with $f\subset\overline{g}$, it is no restriction to remove the connected component of $H_J$ disjoint from $f$ for any $J\subset I$, or equivalently to assume that such $H_J$ are connected. It follows that the faces appearing in the definition of $\delta(f)$ are in one-to-one correspondence with the uples $J\subset I$ so they can be indexed by them and eventually omitted in the sum defining  $\delta^*(f)$.  It follows that, $\epsilon_I$ denoting an orientation of $N(f)$,
\begin{eqnarray*}
 \delta^{2k}\circ \delta^1(\epsilon_I) &=& \sum_{|J|=2k}\sum_{1\le i\le N}  e_J\lrcorner e_i \lrcorner\epsilon_I
  = \sum_{|J|=2k+1}\sum_{l=1}^{2k+1} e_{j_1}\lrcorner\cdots \widehat{e_{j_l}}\lrcorner\cdots \lrcorner e_{j_{2k+1}}\lrcorner e_{j_l}\lrcorner \epsilon_I  \\
   &=&  \sum_{|J|=2k+1}\sum_{l=1}^{2k+1} (-1)^{l-1}e_J\lrcorner\epsilon_I 
   =  \sum_{|J|=2k+1}e_J\lrcorner\epsilon_I=\delta^{2k+1}(\epsilon_I).
\end{eqnarray*}
The equality $\delta^{2k+1} = \delta^1\circ \delta^{2k}$ is obtained in the same way.
\end{proof}

\begin{proof}[of Proposition \ref{prop:contracting-delta}]
Let us set $N = \sum_{k\ge 0} \delta^{2k}$ and  $h=\Id +N$. Using the lemma, we get: 
\begin{equation}
    \delta^{\pcn} = \delta^1\circ h=h\circ \delta^1. 
\end{equation}
Since $N$ is   nilpotent, the map $h$ is invertible with inverse given by the finite sum
\[
 h^{-1}=\sum_{j\ge 0}(-1)^jN^j.
\]
This proves that $\delta^1(x)=0$ if and only if $\delta^{\pcn}(x)=0$ and that $x=\delta^1(y)$ if and only if $x=\delta^{\pcn}(y')$ for some $y,y'$ as well. The proposition follows.
\end{proof}

\bigskip \noindent The differential $\delta^1$ is of course much simpler to handle than $\delta^{\pcn}$.

\bibliographystyle{plain} 
\bibliography{CornersAindex} 

\end{document}